\documentclass[11pt, reqno]{amsart}
\usepackage{amsfonts,latexsym, enumerate}
\usepackage{amsmath}
\usepackage{amscd}
\usepackage{float,amsmath,amssymb,mathrsfs,bm,multirow,graphics}
\usepackage[dvips]{graphicx}
\usepackage[percent]{overpic}

\usepackage{tikz}
\usepackage{pict2e}
\usetikzlibrary{arrows,calc,chains, positioning, shapes.geometric,shapes.symbols,decorations.markings,arrows.meta}
\usetikzlibrary{patterns,angles,quotes}
\usetikzlibrary{decorations.markings}

\numberwithin{equation}{section}
\newcommand{\R}{{\Bbb R}}

\newcommand{\C}{{\Bbb C}}

\newcommand{\Z}{{\Bbb Z}}

 \newcommand\beq{\begin{equation}}
 \newcommand\eeq{\end{equation}}

\DeclareMathOperator{\im}{Im}
\DeclareMathOperator{\re}{Re}

\DeclareMathOperator{\arcsinh}{arcsinh}

\newcommand{\res}{\text{\upshape Res\,}}

\newcommand{\Li}{\text{\upshape Li}}

\newcommand\lb{\left(}
\newcommand\rb{\right)}

\def\XXint#1#2#3{{\setbox0=\hbox{$#1{#2#3}{\int}$}
\vcenter{\hbox{$#2#3$}}\kern-.5\wd0}}

\allowdisplaybreaks

\newtheorem{theorem}{Theorem}[section]
\newtheorem{proposition}[theorem]{Proposition}

\newtheorem{lemma}[theorem]{Lemma}
\newtheorem{definition}[theorem]{Definition}

\newtheorem{figuretext}{Figure}

\usepackage[colorlinks=true]{hyperref}
\hypersetup{urlcolor=blue, citecolor=red, linkcolor=blue}

\title{Semiclassical limit of a non-polynomial  \\ $q$-Askey scheme}
\author{Jonatan Lenells}
\address{Department of Mathematics, KTH Royal Institute of Technology \\ S-100 44 Stockholm, Sweden.}
\email{jlenells@kth.se}

\author{Julien Roussillon}
\address{Department of Mathematics and Systems Analysis \\ P.O. Box 11100, FI-00076, Aalto University, Finland.}
\email{julien.roussillon@aalto.fi}

\begin{document}

\begin{abstract}
We prove 
a semiclassical asymptotic formula for the two elements $\mathcal M$ and $\mathcal Q$ lying at the bottom of the recently constructed non-polynomial hyperbolic $q$-Askey scheme. We also prove that the corresponding exponent is a generating function of the canonical transformation between pairs of Darboux coordinates on the monodromy manifold of the Painlev\'e I and $\textrm{III}_3$ equations, respectively.  Such pairs of coordinates characterize the asymptotics of the tau function of the corresponding Painlev\'e equation.
We conjecture that the other members of the non-polynomial hyperbolic $q$-Askey scheme yield generating functions associated to the other Painlev\'e equations in the semiclassical limit.
\end{abstract} 

\maketitle


\section{Introduction}

\subsection{Introduction}

The ($q$-)Askey scheme provides a way of classifying the classical (basic) hypergeometric orthogonal polynomials. Recently, the authors of the present article constructed a non-polynomial generalization of the $q$-Askey scheme \cite{LR24}. Any element of this scheme admits an integral representation, whose integrand is expressed in terms of Ruijsenaars' hyperbolic function \cite{R1997}, or, equivalently, in terms of Fadeev's quantum dilogarithm \cite{FK1994}, Woronowitz quantum exponential \cite{W2000}, or Kurokawa's double sine function \cite{K1991}. Moreover, it was shown in \cite{LR24} that each member of the scheme admits a certain limit towards one or two elements of the $q$-Askey scheme. 

On the other hand, Painlev\'e equations are nonlinear, second-order ordinary differential equations (ODEs) which describe monodromy preserving deformations of linear systems of second order ODEs with prescribed singularities.  The (generalized) monodromy data of such linear systems are encoded in the corresponding monodromy manifold.
Moreover, Painlev\'e equations can be formulated as non-autonomous Hamiltonian systems, whose Hamiltonian is closely related to Jimbo--Miwa--Ueno's (JMU) isomonodromic tau function \cite{JMU}. An important analytic issue is to understand the behavior of the JMU tau function when two singularities of the initial linear system collide. In the case of Painlev\'e VI, the first few terms in the asymptotic expansion of the tau function were constructed explicitly and parametrized in terms of the monodromy data by Jimbo in \cite{Jim} utilizing Riemann-Hilbert techniques. Remarkably, the Kyiv formula, conjectured by Gamayun, Iorgov, and Lisovyy in \cite{IoLT} and proved by Gavrylenko and Lisovyy in \cite{GL_PVI},  upgrades Jimbo's formula to a full asymptotic expansion. Surprisingly, the Kyiv formula admits a Fourier structure whose Fourier coefficients are the so-called four-point Virasoro conformal blocks at central charge $c=1$. In particular, the Fourier expansion is written explicitly in terms of suitable pairs of Darboux coordinates on the Painlev\'e VI monodromy manifold. More generally, it is believed that the asymptotic expansion of the tau function of any Painlev\'e equation around any critical point or direction possesses a Fourier structure (see \cite{BLMST} for many explicit examples). 

Of special importance is the fact that the JMU tau function is defined up to a multiplicative factor.  The \emph{connection constant problem} is the problem of relating ratios of such multiplicative factors corresponding to different critical points or directions. 
The only connection constant problems for generic Painlev\'e tau functions that have been solved rigorously are associated with Painlev\'e VI \cite{ILP} (first conjectured in \cite{IoLT}), Painlev\'e III$_3$ \cite{IP} (first conjectured in \cite{ILT}), and Painlev\'e I \cite{LisoRou}. In all these cases, the connection constant is closely related to the generating function of the canonical transformation between two pairs of Darboux coordinates that characterize the asymptotic expansions of the tau function. 

In this article, we study the relationship between the semiclassical limit of the non-polynomial $q$-Askey scheme of \cite{LR24} and Painlev\'e equations. We consider the two elements $\mathcal M(b,x,y)$ and $\mathcal Q(b,x,y)$ at the bottom level of the scheme in the semiclassical limit. This means that we study the behavior of $\mathcal{M}(b,x/b,y/b)$ and $\mathcal{Q}(b,x/b,y/b)$ as $b$ tends to $0$. The limit $b \to 0$ is called the semiclassical limit because the parameter $b$ characterizes two quantum deformation parameters $q=e^{i\pi b^2}$ and $\tilde q = e^{i\pi b^{-2}}$. We prove that $\phi \in \{\mathcal M, \mathcal Q\}$ satisfies
\begin{align}\label{phiasymptotics}
& \phi \bigg(b,\frac{x}b,\frac{y}b\bigg) =  g_\phi(x,y) e^{\frac{f_\phi(x,y)}{b^2}} \lb 1 + O (b) \rb \qquad \text{as} \; b \to 0^+,
\end{align}
where $g_\phi$ and $f_\phi$ are explicit functions. The proof of (\ref{phiasymptotics}) is based on a rigorous saddle point analysis of the associated integral representation. Moreover, we prove a simple formula which relates $f_{\mathcal M}$ (resp. $f_{\mathcal Q}$) to the generating function of the canonical transformation between two suitable pairs of Darboux coordinates on the Painlev\'e $\textrm{I}$ (resp. Painlev\'e $\textrm{III}_3$) monodromy manifold. 

We initially derived the relation to Painlev\'e generating functions in an informal manner as follows. On the one hand, it was proved in \cite{LR24} that the functions $\mathcal M(b,x,y)$ and $\mathcal Q(b,x,y)$ satisfy two pairs of difference equations in the variables $x$ and $y$. On the other hand, by definition the generating functions of interest satisfy a pair of differential equations. We observed that the semiclassical limit of the difference equations satisfied by $\mathcal M$ and $\mathcal Q$ formally leads to the differential equations satisfied by the corresponding generating functions (see Appendix \ref{appendix} for more details).

More generally, we conjecture that the semiclassical limit of the other elements of the scheme are related to generating functions for Painlev\'e equations according to the diagrams of Figure \ref{nonpolynomialscheme}.

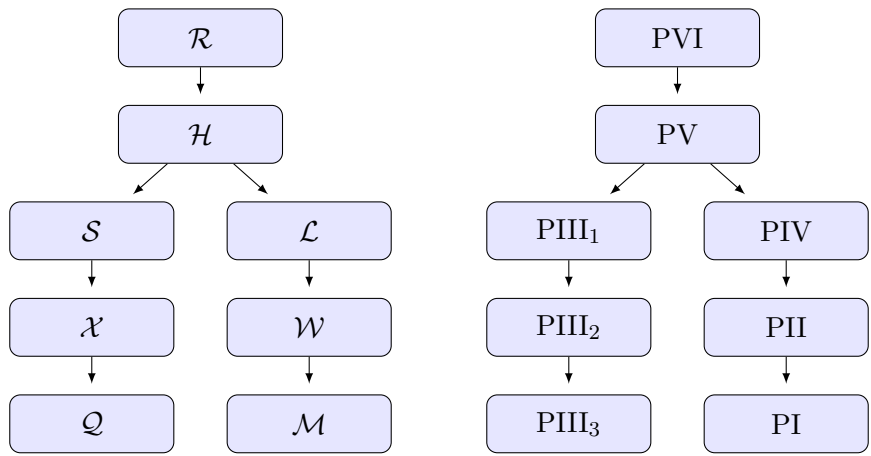
\begin{figure}[h!]
\tikzstyle{block} = [rectangle, draw, fill=blue!10, 
    text width=5em, text centered, rounded corners, minimum height=2em]
\tikzstyle{line} = [draw, -latex, shorten >= 4pt, shorten <= 0pt]
   \centering
\begin{tikzpicture}[node distance = .1cm, auto,scale=0.1]
    \node [block] (AW) {$\mathcal{R}$};
    \node [block, below =0.5cm of AW] (HahnJacobi) {$\mathcal{H}$};
\node [block, below left=.5cm and -.75cm  of HahnJacobi] (SalamLittleJacobi) {$\mathcal{S}$};
\node [block, below right=.5cm and -.75cm  of HahnJacobi] (BigLaguerre) {$\mathcal{L}$};
\node [block, below =.5cm of BigLaguerre] (LittleLaguerre) {$\mathcal{W}$};
\node [block, below=.5cm of LittleLaguerre] (M) {$\mathcal{M}$};
\node [block, below =.5cm of SalamLittleJacobi] (ContinuousBigHermite) {$\mathcal{X}$};
\node [block, below =.5cm of ContinuousBigHermite] (BigHermite) {$\mathcal{Q}$};

    \path [line] (AW) -- (HahnJacobi);
    \path [line] (HahnJacobi) -- (SalamLittleJacobi);
    \path [line] (HahnJacobi) -- (BigLaguerre);
    \path [line] (BigLaguerre) -- (LittleLaguerre);
     \path [line] (SalamLittleJacobi) -- (ContinuousBigHermite);
     \path [line] (ContinuousBigHermite) -- (BigHermite);
     \path [line] (LittleLaguerre) -- (M);

\end{tikzpicture}
\hspace{1cm}
\begin{tikzpicture}[node distance = .1cm, auto,scale=0.1]
    \node [block] (AW) {PVI};
    \node [block, below =0.5cm of AW] (HahnJacobi) {PV};
\node [block, below left=.5cm and -.75cm  of HahnJacobi] (SalamLittleJacobi) {PIII$_1$};
\node [block, below right=.5cm and -.75cm  of HahnJacobi] (BigLaguerre) {PIV};
\node [block, below =.5cm of BigLaguerre] (LittleLaguerre) {PII};
\node [block, below=.5cm of LittleLaguerre] (M) {PI};
\node [block, below =.5cm of SalamLittleJacobi] (ContinuousBigHermite) {PIII$_2$};
\node [block, below =.5cm of ContinuousBigHermite] (BigHermite) {PIII$_3$};

    \path [line] (AW) -- (HahnJacobi);
    \path [line] (HahnJacobi) -- (SalamLittleJacobi);
    \path [line] (HahnJacobi) -- (BigLaguerre);
    \path [line] (BigLaguerre) -- (LittleLaguerre);
     \path [line] (SalamLittleJacobi) -- (ContinuousBigHermite);
     \path [line] (ContinuousBigHermite) -- (BigHermite);
     \path [line] (LittleLaguerre) -- (M);

\end{tikzpicture}
\caption{The elements of the non-polynomial $q$-Askey scheme and the corresponding Painlev\'e equations} 
\label{nonpolynomialscheme}
\end{figure}

\subsection{Relations to earlier works}

We emphasize that the conjecture above is certainly true in the case of the top element $\mathcal R$ of the scheme, due to earlier results of Teschner and Vartanov in \cite{TV}. In fact, their result motivated us to investigate the semiclassical limit of the non-polynomial $q$-Askey scheme in the first place. More precisely, the connection constant for the Painlev\'e VI tau function is closely related to a certain generating function which, interestingly, coincides with the volume of a hyperbolic tetrahedron \cite{IoLT}. Teschner and Vartanov showed in \cite{TV} at a theoretical physics level of rigor that the Virasoro fusion kernel, which is a renormalized version of the element $\mathcal R$ \cite{R2021}, reduces to such a generating function in the semiclassical limit. 

\subsection{Outlook}

As mentioned above, only three connection constants for Painlev\'e tau functions have been rigorously obtained in the literature. The case of Painlev\'e V was conjectured in \cite{LNR}, and we expect it to be closely related to the semiclassical limit of the function $\mathcal H$. To our knowledge, apart from a special case of Painlev\'e II studied in \cite{ILP}, the connection constants for the other Painlev\'e equations remain unknown. According to the conjecture above, we believe that the semiclassical limit of the non-polynomial $q$-Askey scheme can help determine these constants.

Moreover, the results of the present paper naturally suggest that elements of the non-polynomial $q$-Askey scheme play an important role in the quantization of Painlev\'e monodromy manifolds. In this context, it is worth noting that Mazzocco discovered a profound connection between the $q$-Askey scheme and Painlev\'e equations in \cite{Mazz16}. Specifically, each Painlev\'e monodromy manifold admits a Poisson bracket that can be quantized. The resulting quantum algebras have a representation on the space of polynomials, and the eigenspaces in this representation are spanned by elements of the $q$-Askey scheme. On the other hand, the authors of the present paper studied in \cite{LR24} the relation between the non-polynomial $q$-Askey scheme and the $q$-Askey scheme. One key feature is that there exists a limit in which the two pairs of difference equations satisfied by elements of the non-polynomial scheme reduce to one difference and one three-term recurrence relation satisfied by elements of the $q$-Askey scheme. It would be interesting to understand how Mazzocco's discovery generalizes to the non-polynomial case. 

Let us finally mention geometric results of Teschner \cite{T1,T2} regarding the top element $\mathcal R$ of the non-polynomial $q$-Askey scheme and the quantization of Teichmüller spaces.

\subsection{Organization of the paper}

The paper is organized as follows. In Section \ref{section2}, we state the main results concerning the semiclassical limits of $\mathcal M$ and $\mathcal Q$ and their relation to the generating functions associated to Painlev\'e I and Painlev\'e III$_3$, respectively. In Section \ref{section3}, we derive an asymptotic formula for Ruijsenaars' hyperbolic gamma function which is important for the proofs. In Sections \ref{section4} and \ref{section5}, we prove the main results stated in Section \ref{section2} by a rigorous saddle point analysis. Finally, in Appendix \ref{appendix} we formally describe the semiclassical limit of the difference equations satisfied by the function $\mathcal M$ which give rise to differential equations satisfied by the Painlev\'e I generating function.

\subsection{Notation}
Unless stated otherwise, the principal branch will be used for all complex powers, all logarithms, and all dilogarithms. More precisely, this means that $z^\alpha := e^{\alpha \ln z}$ where the logarithm $\ln z$ satisfies $\im \ln z \in (-\pi, \pi]$ and that the dilogarithm $\Li_2(z)$ is the analytic function of $z \in \C \setminus [1, +\infty)$ defined by
\begin{align}\label{dilogdef}
\Li_2(z) = -\int_0^{z} \frac{\ln(1-u)}{u} du \qquad \text{for $z \in \C \setminus [1,+\infty)$},
\end{align}
where the curve of integration is any contour from $0$ to $z$ avoiding the branch cut $[1,+\infty)$.
Unless stated otherwise, we assume that $b$ is a strictly positive real number.
The notation $b \to 0^+$ indicates the limit as $b \in \R$ approaches zero from above.

\subsection*{Acknowledgements}
J.L. was supported by the Swedish Research Council, Grant No. 2021-03877.
J.R. was supported by the Academy of Finland Centre of Excellence Programme, Grant No. 346315 entitled ``Finnish centre of excellence in Randomness and STructures (FiRST)''.

\section{Main results} \label{section2}

\subsection{Definition of the hyperbolic gamma function $s_b(z)$}
The elements $\mathcal M$ and $\mathcal Q$ of the non-polynomial scheme admit contour integral representations, whose integrands involve the function $s_b(z)$ defined by
\begin{equation}\label{defsb}
s_b(z)=\operatorname{exp}{\left[  i \int_0^\infty \left(\frac{\operatorname{sin}{2yz}}{2\operatorname{sinh}(b^{-1}y)\operatorname{sinh}(by)}-\frac{z}{y}\right)  \frac{dy}{y} \right]}, \qquad |\im z|<\frac{Q}{2},
\end{equation}
where $Q := b + b^{-1}$. The function $s_b(z)$ is related to Ruijsenaars' hyperbolic gamma function $G$ in \cite{R1997} by
\begin{align}\label{sbgbGE}
s_b(z) = G(b, b^{-1}; z).
\end{align}
It follows from \eqref{sbgbGE} and the results in \cite{R1997} that $s_b(z)$ is a meromorphic function of $z \in \mathbb{C}$ with zeros $\{z_{m,l}\}_{m,l=0}^\infty$ and poles $\{p_{m,l}\}_{m,l=0}^\infty$ located at 
\begin{equation}\label{polesb}
\begin{split}
&z_{m,l}=\frac{i Q}{2}+i m b +il b^{-1}, \qquad m,l = 0,1, 2,\dots, \qquad (\text{zeros}), 
	\\
&p_{m,l}=-\frac{i Q}{2} -i m b-il b^{-1}, \qquad m,l = 0,1, 2,\dots, \qquad (\text{poles}).
\end{split}
\end{equation}
The multiplicity of the zero $z_{m,l}$ in (\ref{polesb}) is given by the number of distinct pairs $(m_i,l_i) \in \mathbb{Z}_{\geq 0} \times \mathbb{Z}_{\geq 0}$ such that $z_{m_i,l_i}=z_{m,l}$. 
The pole $p_{m,l}$ has the same multiplicity as the zero $z_{m,l}$.


\subsection{The function $\mathcal  M$ and the Painlev\'e $\textrm I$ generating function}
Let $\Omega_{\mathcal M}$ be the domain
\beq \label{DMdef}
\Omega_\mathcal{M} := \{(\zeta, \omega) \in \mathbb C^2 \,|\, \im \omega<Q/2,\; \im(\zeta+\omega) >0 \} \setminus (\Delta_\zeta \times \mathbb C),
\eeq
where
$$\Delta_{\zeta} := \{\tfrac{iQ}2+ibm+ilb^{-1}\}_{m,l=0}^\infty \cup \{-\tfrac{iQ}2-imb-ilb^{-1} \}_{m,l=0}^\infty.$$
The function $\mathcal M$ is an analytic function of $(\zeta,\omega) \in \Omega_\mathcal{M}$ given by \cite[Eq. (10.5)]{LR24}
\beq \label{defM}
\mathcal{M}(b,\zeta,\omega) = P_\mathcal{M}(\zeta) \displaystyle \int_{\mathsf{M}} I_\mathcal{M}(x,\zeta,\omega) dx \qquad \text{for $(\zeta,\omega) \in \Omega_\mathcal{M}$},
\eeq
where
\begin{align}
P_\mathcal{M}(\zeta) = s_b(\zeta ),  \qquad
 I_\mathcal{M}(x,\zeta,\omega) = e^{i \pi  x \left(\zeta -\frac{i Q}{2}+2 \omega \right)} \frac{s_b(x-\zeta ) }{s_b\big(x+\frac{i Q}{2}\big)},
\end{align}
and the contour $\mathsf{M}$ is any curve from $-\infty$ to $+\infty$ which separates the increasing from the decreasing sequence of poles. If $\zeta \in \mathbb R$ and $\im \omega \in (0,Q/2)$, then the contour $\mathsf{M}$ can be any curve from $-\infty$ to $+\infty$ lying within the strip $\im x \in (-Q/2,0)$.

We also introduce the set of differential equations characterizing the Painlev\'e I generating function.

\begin{definition} \label{defgeneratingPI}
A solution $Y(\zeta,\omega)$ of the differential equations
\begin{align}
\label{diffeqWzeta} & \frac{\partial Y(\zeta,\omega)}{\partial \zeta} =  \ln\lb 1-e^{2i\pi(\zeta + \omega)}\rb, \\
\label{diffeqWomega} &  \frac{\partial Y(\zeta,\omega)}{\partial \omega} = \ln \lb \frac{1-e^{2i\pi(\zeta+\omega)}}{1 - e^{2i\pi \omega}} \rb,
\end{align}
is called a Painlev\'e I generating function.
\end{definition}

Since the differential form 
$$\ln\lb 1-e^{2i\pi(\zeta + \omega)}\rb d\zeta + \ln \lb \frac{1-e^{2i\pi(\zeta+\omega)}}{1 - e^{2i\pi \omega}} \rb d\omega$$
is closed, the solution $Y(\zeta,\omega)$ of (\ref{diffeqWzeta})--(\ref{diffeqWomega}) exists and is unique up to a constant (in spite of this mild non-uniqueness, we often speak of {\it the} generating function). 

We are now ready to state our first main result, which describes the semiclassical behavior of the function $\mathcal M$ and its relation to the Painlev\'e I generating function.

\begin{theorem}[Semiclassical limit of $\mathcal{M}$]\label{Mth}
For any $\zeta \in \R$ and any $\omega \in \C$ such that $0 < \im \omega < 1/2$, it holds that
\begin{align}\label{Mfinalasymptotics}
\mathcal{M}\bigg(b, \frac{\zeta}{b}, \frac{\omega}{b}\bigg) 
= \sqrt{\frac{\coth (\pi  (\zeta +\omega ))+1}{2}} e^{\frac{f_{\mathcal{M}}(\zeta, \omega)}{b^2}} \big(1 + O(b)\big) \qquad \text{as $b \to 0^+$},
\end{align}
where the function $f_{\mathcal{M}}(\zeta, \omega)$ is defined by
\begin{align}\nonumber
f_{\mathcal{M}}(\zeta, \omega)
=& -\frac{ \Li_2\left(-e^{2 \pi  \zeta }\right) + \Li_2\left(-e^{-\pi  \zeta} \frac{ \cosh (\pi \omega)}{\sinh(\pi  (\zeta +\omega ))}\right) - \Li_2\left(e^{\pi  \zeta } \frac{ \cosh (\pi \omega)}{\sinh(\pi  (\zeta +\omega ))}\right)}{2 \pi i} 
	\\ \label{finthmM} 
& +i \pi  \zeta ^2+ i\bigg(\omega - \frac{i}{2}\bigg)  \ln \left(e^{\pi  \zeta }\frac{ \cosh (\pi \omega)}{\sinh(\pi  (\zeta +\omega ))}\right) +\frac{i \pi }{6}
\end{align}
and the principal branch is used for all functions. Moreover, let $Y(\zeta,\omega)$ be the function defined for $\zeta \in -\frac{1}{2} + i\R$ and $\omega$ in the strip $\re \omega \in (1/2, 1)$ by
\begin{equation} \label{identificationPIgenfunc}
Y(\zeta,\omega) = f_{\mathcal{M}}\bigg( i\zeta + \frac{i}2,i\omega-\frac{i}2 \bigg) + 2i\pi\zeta\omega+i\pi \zeta(\zeta-1)
\end{equation}
and extended to other values of $(\zeta, \omega)$ by analytic continuation. Then $Y(\zeta,\omega)$ is a Painlev\'e I generating function.
\end{theorem}

The proof of Theorem \ref{Mth} is presented in Section \ref{section4}.  The proof of \eqref{Mfinalasymptotics} consists of a rigorous saddle point analysis, whereas the proof that \eqref{identificationPIgenfunc} is a Painlev\'e I generating function consists of a direct verification of the differential equations \eqref{diffeqWzeta} and \eqref{diffeqWomega}.

We now discuss the origin of Definition \ref{defgeneratingPI}. The standard form of Painlev\'e I reads
\begin{equation} \label{PIeq}
\frac{d^2 q(t)}{dt^2 }= 6 q^2 + t.
\end{equation}
This equation can be written in the Hamiltonian form $\frac{dq(t)}{dt} = \frac{\partial H}{\partial p}$, $\frac{dp(t)}{dt} = - \frac{\partial H}{\partial q}$ with the Hamiltonian 
\begin{equation}
H = \frac{p^2}2 - 2q^3 - tq.
\end{equation}
The Painlev\'e I tau function is then defined up to a factor independent of $t$ as
\begin{equation}
\frac{d \ln \tau}{d t} = H.
\end{equation}
Kapaev studied in \cite{Kap1,Kap2} the asymptotics of $q(t)$ as $|t| \to \infty$ along directions $\mathcal R_k = \{\text{arg}\; t = \pi - \frac{2k\pi}5\}$ with $k \in \mathbb Z/5\mathbb Z$. Translated in terms of the tau function, his results imply that the leading asymptotics of $\tau(t)$ as $t \in \mathcal R_k$ tends to infinity is
\begin{align} \label{asytau}
\tau(t) =  \mathcal C_k \; 48^{-\frac{\nu_k^2}2} (2\pi)^{-\frac{\nu_k}2} e^{-\frac{i\pi \nu_k^2}4} G(1+\nu_k) \; x^{-\frac{\nu_k^2}2-\frac{1}{60}} e^{\frac{x^2}{45} + \frac45 i \nu_k x} (1+o(1)), 
\end{align}
where $x=24^{\frac14} \big( e^{\frac{2i\pi k}5 - i\pi}t \big)^\frac54$.  The parameters $\nu_k$, $k \in \mathbb Z/5\mathbb Z$ are related to the monodromy data of the linear system associated with Painlev\'e I. In fact, the space of generic monodromy data $\mathcal S$ is two-dimensional because of the following cyclic relations:
\begin{equation} \label{cyclicrelation}
e^{2i\pi \nu_{k-1}} e^{2i\pi \nu_{k+1}} = 1 - e^{2i\pi \nu_k}, \qquad k \in \mathbb Z/5\mathbb Z.
\end{equation}
Thus, any two of the five parameters $\nu_1,\dots,\nu_5$ provides a set of local coordinates on $\mathcal S$ and may be thought of as Painlev\'e I integration constants. Moreover,  the choice of a pair of parameters fixes all subleading corrections in \eqref{asytau}. It turns out that the pair $(\nu_k,\nu_{k+1})$ is the most convenient 
one to parametrize the asymptotics of the tau function as $|t| \to \infty$ along the ray $\mathcal R_k$. For more details, the reader is referred to \cite{LisoRou}. 

Of special importance is the constant factor $\mathcal C_k$ in \eqref{asytau}. It is not uniquely defined, since the tau function is defined only via a logarithmic derivative. However, the ratio of two constant factors $\Upsilon_{k,k'} = \mathcal C_{k'}/ \mathcal C_k$ is a well-defined function of the monodromy. It was shown in \cite{LisoRou} that $\ln(\Upsilon_{k,k'})$ coincides with the generating function of the canonical transformation between the two pairs of coordinates $(\nu_k,\nu_{k+1})$ and $(\nu_{k'},\nu_{k'+1})$. 

The generating function which leads to Definition \ref{defgeneratingPI} is $\ln(\Upsilon_{k-1,k+2})$.  Its explicit formula can be found utilizing \cite[Theorem 3.6]{LisoRou}. By definition, it satisfies the differential equations
\begin{align*}
& \frac{d\ln(\Upsilon_{k-1,k+2})}{d \nu_k} = 2i\pi \nu_{k-1}, \qquad \frac{d\ln(\Upsilon_{k-1,k+2})}{d \nu_{k+3}} = - 2i\pi \nu_{k+2}.
\end{align*}
Moreover, the cyclic relation \eqref{cyclicrelation} implies that
\begin{align} 
 e^{2i\pi \nu_{k-1}} = 1-e^{2i\pi (\nu_k+\nu_{k+3})}, \qquad e^{2i\pi \nu_{k+2}} = \frac{1-e^{2i\pi \nu_{k+3}}}{1-e^{2i\pi (\nu_k+\nu_{k+3})}}.
\end{align}
Assuming that $\arg(2i\pi \nu_{k-1})$ and $\arg(2i\pi \nu_{k+2})$ lie in the range $(-\pi,\pi)$ and using the principal branch for the logarithm, the differential equations can be recast in terms of $\nu_k$ and $\nu_{k+3}$ only as
\begin{align*}
 \frac{d\ln(\Upsilon_{k-1,k+2})}{d \nu_k} = \ln \lb 1-e^{2i\pi (\nu_k+\nu_{k+3})} \rb, \quad \frac{d\ln(\Upsilon_{k-1,k+2})}{d \nu_{k+3}} = \ln\bigg(\frac{1-e^{2i\pi (\nu_k+\nu_{k+3})}}{1-e^{2i\pi \nu_{k+3}}}\bigg).
\end{align*}
These equations are equivalent to the ones of Definition \ref{defgeneratingPI}.

\subsection{The function $\mathcal Q$ and the Painlev\'e $\textrm{III}_3$ generating function}
Our second main result concerns the semiclassical asymptotics of the function $\mathcal Q$ and its relation to the Painlev\'e III$_3$ generating function.

We first recall the definition of $\mathcal Q$ from \cite{LR24}. Let $\Omega_\mathcal{Q}$ be the domain \beq
\Omega_\mathcal{Q} := \mathbb C \times \big(\{\mu \in \C \, | \, \im \mu <Q/2 \} \backslash \Delta_\mu\big),
\eeq
where
\beq \label{Detadef}
\Delta_\mu := \{-\tfrac{iQ}2 -imb - ilb^{-1} \}_{m,l=0}^\infty.
\eeq
The function $\mathcal{Q}$ admits the following integral representation \cite[Eq. (7.6)]{LR24}:
\beq\label{defQ}
\mathcal{Q}(b,\sigma_s,\mu) = P_\mathcal{Q}(\sigma_s,\mu) \displaystyle \int_{\mathsf{Q}} I_\mathcal{Q}(x,\sigma_s,\mu) dx \qquad \text{for $(\sigma_s,\mu) \in \Omega_\mathcal{Q}$},
\eeq
where 
\begin{align}
\label{PQ} & P_\mathcal{Q}(\sigma_s,\mu) = e^{i \pi  (\frac{1}{6}+\frac{7 Q^2}{24}-\frac{\mu^2}{2}+i \mu  Q- \sigma_s^2)} s_b(\mu), \\
\label{IQ} & I_\mathcal{Q}(x,\sigma_s,\mu) = e^{-i\pi x^2} e^{2i\pi x(\mu-\frac{iQ}2)} s_b(x+\sigma_s) s_b(x-\sigma_s),
\end{align}
and the integration contour $\mathsf{Q}$ is any curve from $-\infty$ to $+\infty$ passing above the points $x= \pm \sigma_s - iQ/2$ such that its right tail satisfies
\beq
\im x + \tfrac{Q}4 - \tfrac{ \im \mu}2 < -\delta \qquad
\text{for all $x \in \mathsf{Q}$ with $\re x$ sufficiently large},
\eeq
for some $\delta>0$.  If $(\sigma_s,\mu) \in \mathbb R^2$, then $\mathsf{Q}$ can be any curve from $-\infty$ to $+\infty$ lying within the strip $\im x \in (-Q/2,-Q/4-\delta)$.

We next define the Painlev\'e III$_3$ generating function, first introduced in \cite{ILT}.

\begin{definition} \label{defgeneratingfuncPIII3}
The Painlev\'e III$_3$ generating function $W(\sigma,\nu)$ is defined by
\begin{equation} \label{defW}
8\pi^2 W(\sigma,\nu) = \Li_2 \lb - e^{2i\pi(\sigma+\eta-\frac{i \nu}2)}\rb  + \Li_2 \lb - e^{-2i\pi(\sigma+\eta+\frac{i \nu}2)}\rb - (2\pi \eta)^2 + (\pi \nu)^2,
\end{equation}
where $\eta = \eta(\sigma, \nu)$ is given by
\begin{align}\label{etasigmanudef}
\eta(\sigma, \nu) = \frac{1}{2\pi} \arcsin(e^{\pi \nu} \sin(2\pi \sigma))
\end{align}
and the branches are fixed as follows: the principal branch is used for the functions $\Li_2(\cdot)$ and $\arcsin(\cdot)$ when $0<\eta<\sigma<\frac14$ and $\nu \in \mathbb R_{<0}$, and the right-hand side of (\ref{defW}) is extended to other values of $(\sigma, \nu)$ by analytic continuation.
\end{definition}

The following theorem is our second main result.

\begin{theorem}[Semiclassical limit of $\mathcal{Q}$]\label{Qth}
For any $(\sigma_s,\mu) \in \R^2$ such that $e^{2 \pi  \mu} > \sinh^2(2 \pi \sigma_s)$, it holds that
\begin{align}\label{Qfinalasymptotics}
\mathcal{Q}\bigg(b, \frac{\sigma_s}{b}, \frac{\mu}{b}\bigg) 
= \frac{e^{\frac{\pi i}{4}}}{\sqrt{2} (e^{2 \pi  \mu }-\sinh^2(2 \pi  \sigma_s))^{1/4}} e^{\frac{f_{\mathcal{Q}}(\sigma_s, \mu)}{b^2}} \big(1 + O(b)\big) \qquad \text{as $b \to 0^+$},
\end{align}
where the function $f_{\mathcal{Q}}(\sigma_s, \mu)$ is defined by
\begin{align} \nonumber
f_{\mathcal{Q}}(\sigma_s, \mu)
=& -\frac{\Li_2\big(-e^{2 \pi  \mu }\big)}{2 \pi i}
 -\frac{\Li_2\big(e^{-2 \pi 
   \sigma_s} \big(\cosh (2 \pi  \sigma_s)+i \sqrt{e^{2 \pi  \mu }-\sinh^2(2 \pi  \sigma_s)}\big)\big)}{2 \pi i}
   	\\ \label{finQ}
 & -\frac{\Li_2\big(e^{2 \pi  \sigma_s}
   \big(\cosh (2 \pi  \sigma_s)+i \sqrt{e^{2 \pi  \mu }-\sinh^2(2 \pi  \sigma_s)}\big)\big)}{2 \pi i}
   	\\ 
\nonumber & + \bigg(\frac{1}{2}+i \mu \bigg) \ln \Big(-\cosh (2 \pi  \sigma_s)-i \sqrt{e^{2 \pi  \mu }-\sinh^2(2 \pi  \sigma_s)}\Big) -\pi \mu +\frac{5 \pi i}{12}
\end{align}
and the principal branch is used for all functions.
Moreover, the exponent $f$ is related to the Painlev\'e III$_3$ generating function as follows:
\begin{equation}\label{fWrelationexpected}
f_{\mathcal{Q}}(\sigma_s, \mu) = 4i\pi W\bigg(i\sigma_s,-\mu - \frac{i}{2}\bigg) 
 - \frac{\Li_2(-e^{2\pi \mu})}{2\pi i}
  - 2i\pi \sigma_s^2 + \frac{i\pi}4
\end{equation}
for any $(\sigma_s,\mu) \in \R^2$ such that $e^{2 \pi  \mu} > \sinh^2(2 \pi \sigma_s)$.
\end{theorem}

The proof of Theorem \ref{Qth} 
is presented in Section \ref{section5}. 

Let us explain the origin of Definition \ref{defgeneratingfuncPIII3}. The Painlev\'e $\textrm{III}_3$ equation 
\begin{equation} \label{painleveIII3}
\frac{d^2u}{dr^2} + \frac{1}r \frac{du}{dr} + \sin(u) = 0
\end{equation}
is also known as the Painlev\'e $\textrm{III}_3(D_8)$ equation. It is a special case of the third Painlev\'e equation and is also a radial-symmetric reduction of the elliptic sine-Gordon equation. 
The Painlev\'e $\textrm{III}_3$ tau function is defined by \cite{JMU}
\begin{equation}
\frac{d}{dr} \operatorname{ln}{\tau(2^{-12}r^4)} = \frac{r}{4} \left[ \lb \frac{i u_r}2 + \frac{1}{r} \rb^2 + \frac{\cos(u)}{2} \right].
\end{equation}
In particular, the logarithmic derivative $\zeta(t) := t \frac{d}{dt} \ln(\tau(t))$ satisfies 
\begin{equation} \label{equationtauIII3}
(t \zeta'')^2 = 4 (\zeta')^2 (\zeta-t\zeta')-4\zeta',
\end{equation}
where $f'(t) := \frac{d f}{dt}$. Equation \eqref{equationtauIII3} admits a two-parameter family of solutions characterized by the following asymptotic expansion as $t \to 0$:
\begin{equation} \label{shortdistancePIII}
\tau(t) = \sum_{n \in \mathbb Z} e^{4i\pi n \eta} \mathcal F(\sigma+n,t).
\end{equation}
Several remarks are in order. 

\begin{itemize}
\item The function $\mathcal F(\sigma,t)$ is the so-called irregular conformal block at central charge $c=1$ \cite{ILT}.  Thanks to the celebrated AGT correspondence \cite{AGT}, $\mathcal F(\sigma,t)$ admits a representation given by a combinatorial sum over pairs of Young diagrams.

\item The pair of parameters $(\sigma,\eta)$ characterizes the monodromy of the linear system associated with the Painlev\'e $\textrm{III}_3$ equation, in fact, it is a pair of Darboux coordinates on the Painlev\'e $\textrm{III}_3$ monodromy manifold. For more details, the reader is referred to \cite{ILT,IP}. 

\item The expansion \eqref{shortdistancePIII} was conjectured in \cite{GIL} and proved in \cite{LGBessel} by constructing the tau function as a Fredholm determinant of a generalized Bessel kernel. The principal minor expansion of the Fredholm determinant then leads to \eqref{shortdistancePIII}.
\end{itemize}

The leading asymptotics of the tau function as $t\to \infty$ reads
\begin{equation} \label{longdistancePIII3}
\tau(2^{-12} r^4) = \chi(\sigma,\nu) e^{\frac{i \pi \nu^2}4} 2^{\nu^2} (2\pi)^{-\frac{i\nu}2} G(1+i\nu) r^{\frac{\nu^2}2+\frac14} e^{\frac{r^2}{16}+\nu r} \lb 1 + o(1)\rb, \quad r \to \infty.
\end{equation}
Moreover,  \cite[Conjecture 2]{ILT} states that the structure of this asymptotic expansion has a Fourier structure similar to \eqref{shortdistancePIII}:
\begin{equation}
\tau(2^{-12} r^4) = \chi(\sigma,\nu) \sum_{n \in \mathbb Z} e^{4i\pi n \rho} \mathcal G(\nu+in,r), 
\end{equation}
where $\mathcal G(\nu,r)$ itself has an asymptotic expansion as $r\to \infty$, whose coefficients are determined recursively using \eqref{equationtauIII3}. The conjecture also states that the pair $(\nu,\rho)$ is related to the pair $(\sigma,\eta)$ characterizing the expansion \eqref{shortdistancePIII} by
\begin{equation} \label{longdistancePIII} 
e^{\pi \nu} = \frac{\sin(2\pi \eta)}{\sin(2\pi \sigma)}, \qquad e^{4i\pi \rho} = \frac{\sin(2\pi \eta)}{\sin(2\pi(\sigma+\eta))}.
\end{equation}

The connection constant $\chi(\sigma,\nu)$ is the analog of the connection constant $\Upsilon_{k-1,k+2}$ considered above in the case of Painlev\'e I. 
An explicit expression for $\chi(\sigma,\nu)$ was conjectured in \cite{ILT} and proved in \cite{IP} utilizing Riemann-Hilbert techniques. Just like in the case of Painlev\'e I, $\chi$ is closely related to the generating function $W(\sigma,\nu)$ of the canonical transformation between the pairs $(\sigma,\eta)$ and $(\rho,\nu)$. Such a generating function is defined by the following differential equations \cite{ILT}:
\begin{align}
 \eta = \frac{\partial W(\sigma,\nu)}{\partial\sigma} = \frac{1}{2\pi} \operatorname{arcsin}(e^{\pi \nu} \sin(2\pi \sigma)), \qquad
\rho = -i \frac{\partial W(\sigma,\nu)}{\partial \nu}.
\end{align}
A solution of these equations was found in \cite{ILT} and corresponds to \eqref{defW}.

\section{Semiclassical limit of Ruijsenaars' hyperbolic gamma function $s_b$} \label{section3}

In this section, we derive an asymptotic formula for the function $s_b(z/b)$ as $b \to 0^+$. This formula will be crucial for the saddle point analysis of the functions $\mathcal M$ and $\mathcal Q$.
Asymptotic formulas for $s_b(z/b)$ have appeared in the literature (see e.g. \cite[Section 3.3]{DGLZ2009}), but we have not found a formula covering the asymptotic regime of interest to us. (For the proofs of Sections \ref{section4} and \ref{section5} it is important that the formula is valid uniformly with respect to $z$ in a horizontal strip of the form $|\im z| \leq \frac{1}{2}-\delta$.) 

The next lemma will be used to compute the terms of the asymptotic expansion.

\begin{lemma}\label{sinhintlemma}
For any integer $n \geq 0$ and any $z \in \C$ in the horizontal strip $|\im z| < \frac{1}{2}$, it holds that
\begin{align}\label{intLi22n}
\int_{\R + i0} \frac{e^{-2ixz}}{\sinh(x)}x^{2n-2} dx
= 2\Li_{2-2n}(-e^{2\pi z}) (\pi i)^{2n-1}.
\end{align}
\end{lemma}
\begin{proof}
Let $n \geq 0$ be an integer. The left-hand side of (\ref{intLi22n}) is easily seen to be an analytic function of $z$ in the strip $|\im z| < \frac{1}{2}$. The polylogarithm $\Li_{2-2n}$ is analytic in $\C \setminus [1, +\infty)$ if $n = 0$, and in $\C \setminus \{1\}$ if $n \geq 1$. Since the argument $-e^{2\pi z}$ avoids $[1, +\infty)$ for $z$ in the strip $|\im z| < \frac{1}{2}$, it follows that the right-hand side of (\ref{intLi22n}) also is analytic in this strip. Hence it is enough to prove (\ref{intLi22n}) for $z < 0$. 

Suppose that $z < 0$. By using rectangles with corners at the points $\pm \pi M$, $\pm \pi M + (\pi M + \pi/2) i$ and letting the integer $M$ tend to infinity, we can deform the contour to infinity in the upper half-plane. Taking the residue contributions from the simple poles at $\{\pi i m \,|\, m = 1, 2,\dots\}$ into account, this yields
\begin{align}\label{inte2ixzsinh}
\int_{\R + i0} \frac{e^{-2ixz}}{\sinh(x)}x^{2n-2} dx
= 2\pi i \sum_{j=1}^\infty (-1)^je^{2\pi j z}  (\pi i j)^{2n-2}
= 2\Li_{2-2n}(-e^{2\pi z}) (\pi i)^{2n-1},
\end{align}
showing that (\ref{intLi22n}) indeed holds for all $z < 0$ and any integer $n \geq 0$.
\end{proof}

The next proposition contains the asymptotic formula for $\ln s_b(z/b)$ that we will need. We let $B_n$ be the Bernoulli numbers defined by
\begin{align}\label{Bernoullidef}
\frac{x}{e^x -1} = \sum_{n=0}^\infty \frac{B_n x^n}{n!},
\end{align}
where the series converges in the disk $|x| < 2\pi$.

\begin{proposition}[Asymptotics of $\ln s_b(z/b)$]\label{sbprop}
For any integer $N \geq 1$ and any $\delta > 0$, it holds that
\begin{align}\nonumber
\ln s_b\Big(\frac{z}{b}\Big) 
= &\; \frac{\pi i}{2} \Big(\frac{z}{b}\Big)^2 + \frac{\pi i}{24}(b^2 + b^{-2}) - \sum_{n=0}^N \frac{(1- 2^{2n-1})B_{2n} }{(2n)!} \Li_{2-2n}(-e^{2\pi z}) (\pi i b^2)^{2n-1} 
	\\ \label{lnsbzbexpansion}
& + O(b^{4N+2}) \qquad \text{as $b \to 0^+$}
\end{align}
uniformly for $z$ in the horizontal strip 
\begin{align}\label{Sdeltadef}
S_\delta := \bigg\{z \in \C \,\bigg|\, |\im z| \leq \frac{1}{2}-\delta\bigg\}.
\end{align}
\end{proposition}
\begin{proof}
Fix $N \geq 1$, $\epsilon > 0$, and $\delta > 0$.$  $
By \cite[Eq. (B.5)]{TV2013}, we have
$$\ln s_b(z) = \frac{\pi i}{2} z^2 + \frac{\pi i}{24}(b^2 + b^{-2}) + \ln e_b(z),$$
where $\ln e_b(z)$ is defined in the strip $|\im z| < Q/2$ by
$$\ln e_b(z) = -\int_{\R + i0} \frac{e^{-2itz}}{\sinh(bt) \sinh(\frac{t}{b})} \frac{dt}{4t}.$$
Making the change of variables $t = bx$, we find that
$$\ln s_b\Big(\frac{z}{b}\Big) = \frac{\pi i}{2} \Big(\frac{z}{b}\Big)^2 + \frac{\pi i}{24}(b^2 + b^{-2}) - \int_{\R + i0}  \frac{e^{-2ixz}}{\sinh(b^2x) \sinh(x)} \frac{dx}{4x}$$
for $b > 0$ and $z$ in the strip $|\im z| < (b^2 + 1)/2$.
Letting $\tau = b^2$ and using Lemma \ref{sinhintlemma}, we conclude that the expansion in (\ref{lnsbzbexpansion}) will follow if we can prove that
\begin{align}\nonumber
 \int_{\R + i0}  \frac{e^{-2ixz}}{\sinh(\tau x) \sinh(x)} \frac{dx}{4x}
 = &\; \sum_{n=0}^N \frac{(1- 2^{2n-1})B_{2n} }{2 (2n)!} \tau^{2n-1} \int_{\R + i0} \frac{e^{-2ixz}}{\sinh(x)}x^{2n-2} dx
 	\\ \label{fintdef}
& + O(\tau^{2N+1}) \qquad \text{as $\tau \to 0^+$}
\end{align}
uniformly for $z \in S_\delta$.

In order to prove (\ref{fintdef}), we fix $r \in (0, \pi/2)$. The integrand in the left-hand side of (\ref{fintdef}) has poles at the points in $\pi i \Z \cup \frac{\pi i}{\tau} \Z$. Thus, for $\tau \in (0,1)$, we may choose the contour of integration so that it runs along $(-\infty, -r]$, passes above the origin along a clockwise semicircle of radius $r$ centered at $0$, and then proceeds along $[r, +\infty)$. 
We write this contour as $C_\tau \cup D_\tau$, where $C_\tau$ is the part of the contour contained in the disk $\{|x| \leq \frac{\pi}{2\tau}\}$, and $D_\tau = \R \cap \{|x| > \frac{\pi}{2\tau}\}$ is the rest of the contour. Straightforward estimates show that
\begin{align}\nonumber
\bigg|\int_{D_\tau} \frac{e^{-2ixz}}{\sinh(\tau x) \sinh(x)} \frac{dx}{4x}\bigg|
& \leq C
\bigg(\int_{-\infty}^{-\frac{\pi}{2\tau}} + \int_{\frac{\pi}{2\tau}}^{+\infty}\bigg) \frac{e^{2|x \im z|}}{\sinh(\pi/2) e^{|x|}} \frac{dx}{4|x|}
	\\ \label{intDtauestimate}
& \leq C \int_{\frac{\pi}{2\tau}}^{+\infty} e^{x (2 |\im z| - 1)} \frac{dx}{x}
 \leq C \int_{\frac{\pi}{2\tau}}^{+\infty} e^{- 2\delta x} \frac{dx}{x}
 \leq C e^{- \delta \frac{\pi}{\tau}}
\end{align}
uniformly for $\tau \in (0,1)$ and $z \in S_\delta$.
On the other hand, using that $B_{2n+1} = 0$ for $n = 1, 2, \dots$, we obtain from (\ref{Bernoullidef}) that
$$\frac{1}{\sinh w} - \frac{1}{w} 
= \frac{1}{w}\bigg(2\frac{-w}{e^{-w} -1} - \frac{-2w}{e^{-2w} -1} - 1\bigg)
= \sum_{n=1}^\infty \frac{2(1- 2^{2n-1})B_{2n} w^{2n-1}}{(2n)!},
$$
where the series converges in the disk $|w| < \pi$.
In particular, for any $\epsilon > 0$, we have
$$\bigg|\frac{1}{\sinh{w}} - \sum_{n=0}^N \frac{2(1- 2^{2n-1})B_{2n} w^{2n-1}}{(2n)!} \bigg|
\leq C |w|^{2N+1}$$
for all $w$ in the disk $|w| < \pi - \epsilon$. It follows that
$$\bigg|\frac{1}{\sinh{\tau x}} - \sum_{n=0}^N \frac{2(1- 2^{2n-1})B_{2n} (\tau x)^{2n-1}}{(2n)!}\bigg| \leq C |\tau x|^{2N+1}$$
for all $x \in C_\tau$ and all $\tau \in (0,1)$. 
We infer that, for $\tau \in (0,1)$ and $z \in S_\delta$,
\begin{align} \nonumber
& \bigg|\int_{C_\tau}  \frac{e^{-2ixz}}{\sinh(\tau x) \sinh(x)} \frac{dx}{4x}
- \sum_{n=0}^N \frac{(1- 2^{2n-1})B_{2n} \tau^{2n-1}}{2(2n)!} \int_{C_\tau}  \frac{e^{-2ixz}}{\sinh(x)} x^{2n-2} dx\bigg|
	\\ \label{intminussumB2n}
& \leq C\tau^{2N+1} \int_{C_\tau}  \bigg| \frac{e^{-2ixz}}{\sinh(x)} x^{2N} \bigg| dx.
\end{align}
Since, for $\tau \in (0,1)$ and $z \in S_\delta \cap \{\re z \leq 0\}$,
\begin{align}\nonumber
\int_{C_\tau} \bigg| \frac{e^{-2ixz}}{\sinh(x)} x^{2N} \bigg| dx
 \leq &\; C\int_{\{r e^{i\alpha} | 0\leq \alpha \leq \pi\}} \frac{e^{2(\im x) (\re z) + 2(\re x) (\im z)}}{\sinh(x)} x^{2N} dx
	\\ \label{intCesinhx2N}
& + C \int_{r}^{\infty}  \frac{e^{2x |\im z|}}{\sinh(x)}  x^{2N} dx
 \leq C,
\end{align}
we deduce that, as $\tau \to 0^+$,
\begin{align}\label{intCtauCtau}
\int_{C_\tau}  \frac{e^{-2ixz}}{\sinh(\tau x) \sinh(x)} \frac{dx}{4x}
= &\; \sum_{n=0}^N \frac{(1- 2^{2n-1})B_{2n} \tau^{2n-1}}{2(2n)!} \int_{C_\tau}  \frac{e^{-2ixz}}{\sinh(x)} x^{2n-2} dx 
+ O(\tau^{2N+1})
\end{align}
uniformly for $z \in S_\delta \cap \{\re z \leq 0\}$. Up to an error of order $O(e^{- \delta \frac{\pi}{2\tau}} )$, the contour $C_\tau$ in the integral on the right-hand side of (\ref{intCtauCtau}) may be replaced by $\R + i0$. Indeed, by an estimate analogous to the one in (\ref{intDtauestimate}), we have, for $n = 0, 1, \dots, N$,
\begin{align}\label{intDtauestimate2}
\bigg|\int_{D_\tau}  \frac{e^{-2ixz}}{\sinh(x)} x^{2n-2} dx \bigg|
\leq C \int_{\frac{\pi}{2\tau}}^{+\infty} e^{- 2\delta x}  x^{2n-2} dx
\leq C e^{- \delta \frac{\pi}{2\tau}} \int_{\frac{\pi}{2\tau}}^{+\infty} e^{- \delta x} x^{2n-2}  dx
\leq C e^{- \delta \frac{\pi}{2\tau}} 
\end{align}
uniformly for $\tau \in (0,1)$ and $z \in S_\delta$.
On the other hand, by (\ref{intDtauestimate}), the contour $C_\tau$ in the integral on the left-hand side of (\ref{intCtauCtau}) may be replaced by $\R + i0$ with an error of order $O(e^{- \delta \frac{\pi}{\tau}} )$.
It therefore follows from (\ref{intCtauCtau}) that (\ref{fintdef}) holds uniformly for $z \in S_\delta \cap \{\re z \leq 0\}$. 

It remains to prove that (\ref{fintdef}) holds uniformly also for $z \in S_\delta \cap \{\re z \geq 0\}$. To this end, let $C_\tau'$ be the contour $C_\tau$ but with the semicircle of radius $r$ in the upper half-plane replaced by a semicircle of the same radius in the lower half-plane, i.e.,
$$C_\tau' := \Big[-\frac{\pi}{2\tau}, -r\Big] \cup \{re^{i\alpha} \, | \, -\pi \leq \alpha \leq 0\} \cup \Big[r, \frac{\pi}{2\tau}\Big],$$ 
The estimate (\ref{intminussumB2n}) remains true with $C_\tau$ replaced by $C_\tau'$. 
Furthermore, analogously to (\ref{intCesinhx2N}), we have, for $\tau \in (0,1)$ and $z \in S_\delta \cap \{\re z \geq 0\}$,
\begin{align*}
\int_{C_\tau'} \bigg| \frac{e^{-2ixz}}{\sinh(x)} x^{2N} \bigg| dx \leq C,
\end{align*}
and hence, as $\tau \to 0^+$,
\begin{align}\label{intCtauprime}
\int_{C_\tau'}  \frac{e^{-2ixz}}{\sinh(\tau x) \sinh(x)} \frac{dx}{4x}
= &\; \sum_{n=0}^N \frac{(1- 2^{2n-1})B_{2n} \tau^{2n-1}}{2(2n)!} \int_{C_\tau'}  \frac{e^{-2ixz}}{\sinh(x)} x^{2n-2} dx 
+ O(\tau^{2N+1})
\end{align}
uniformly for $z \in S_\delta \cap \{\re z \geq 0\}$. By (\ref{intDtauestimate}) and (\ref{intDtauestimate2}), the contour $C_\tau'$ on the left- and right-hand side of (\ref{intCtauprime}) may be replaced by $\R - i0$ with errors of order $O(e^{- \delta \frac{\pi}{\tau}} )$ and $O(e^{- \delta \frac{\pi}{2\tau}} )$, respectively. Thus
\begin{align}\nonumber
 \int_{\R - i0}  \frac{e^{-2ixz}}{\sinh(\tau x) \sinh(x)} \frac{dx}{4x}
 = &\; \sum_{n=0}^N \frac{(1- 2^{2n-1})B_{2n} }{2 (2n)!} \tau^{2n-1} \int_{\R - i0} \frac{e^{-2ixz}}{\sinh(x)}x^{2n-2} dx
	\\ \label{fintdef2}
& + O(\tau^{2N+1}) \qquad \text{as $\tau \to 0^+$}
\end{align}
uniformly for $z \in S_\delta \cap \{\re z \geq 0\}$.
Deforming the contours in (\ref{fintdef2}) upward through the origin using the relations
$$\underset{x = 0}{\res} \bigg(\frac{e^{-2ixz}}{\sinh(\tau x) \sinh(x)} \frac{1}{4x}\bigg) = -\frac{\tau ^2+12 z^2+1}{24 \tau }$$
and
$$\underset{x = 0}{\res} \frac{e^{-2ixz}}{\sinh(x) x^2} = -\frac{1}{6} - 2z^2, \qquad 
\underset{x = 0}{\res} \frac{e^{-2ixz}}{\sinh(x)} = 1,$$
we find that (\ref{fintdef2}) remains true if the two integration contours are replaced by $\R + i0$. 
This proves that (\ref{fintdef}) holds uniformly for $z \in S_\delta \cap \{\re z \geq 0\}$ and completes the proof of the proposition.
\end{proof}

\section{Semiclassical limit of $\mathcal M$: Proof of Theorem \ref{Mth}} \label{section4}

\subsection{Proof of the asymptotic formula (\ref{Mfinalasymptotics}) for $\mathcal{M}$}
Fix $\zeta \in \R$ and $\omega \in \C$ such that $0 < \im \omega < 1/2$.
The definition \eqref{defM} of the function $\mathcal M$ implies that
\begin{align*}
\mathcal{M}\bigg(b, \frac{\zeta}{b}, \frac{\omega}{b}\bigg) = s_b\bigg(\frac{\zeta}{b}\bigg) \int_{\mathsf{M}} e^{i \pi  x \left(\frac{\zeta}{b} -\frac{i Q}{2}+2 \frac{\omega}{b} \right)} \frac{s_b(x-\frac{\zeta}{b} ) }{s_b\big(x+\frac{i Q}{2}\big)} dx.
\end{align*}
The change of variables $x = t/b$ gives
\begin{align}\label{calMtintegral}
\mathcal{M}\bigg(b, \frac{\zeta}{b}, \frac{\omega}{b}\bigg) 
=&\; 
 \int_{b\mathsf{M}} s_b\bigg(\frac{\zeta}{b}\bigg) e^{i \pi  \frac{t}{b} \left(\frac{\zeta}{b} -\frac{i Q}{2}+2 \frac{\omega}{b} \right)} \frac{s_b(\frac{t}{b}-\frac{\zeta}{b} ) }{s_b\big(\frac{t}{b}+\frac{i Q}{2}\big)} \frac{dt}{b}.
\end{align}
Since $\zeta \in \mathbb R$ and $\im \omega \in (0,1/2)$, the contour $b\mathsf{M}$ can be any curve from $-\infty$ to $+\infty$ lying within the strip $\im t \in (-1/2,0)$. We choose a small $\delta > 0$ and let $b\mathsf{M}$ be a contour independent of $b$ that lies within the slightly smaller strip $\im t \in (-1/2 + \delta, -2\delta)$.
If $S_\delta$ is the strip introduced in (\ref{Sdeltadef}), then $\zeta$, $t-\zeta$, and $t + biQ/2 = t + (1+b^2)i/2$ lie in $S_\delta$ for all $t \in b\mathsf{M}$ and all small enough $b>0$.
Therefore we can apply the asymptotic expansion of $s_b(z/b)$ obtained in Proposition \ref{sbprop} with $N = 1$ to deduce that the integrand in (\ref{calMtintegral}) satisfies 
\begin{align}\label{calMintegrandexpansion}
s_b\bigg(\frac{\zeta}{b}\bigg) e^{i \pi  \frac{t}{b} \left(\frac{\zeta}{b} -\frac{i Q}{2}+2 \frac{\omega}{b} \right)} \frac{s_b(\frac{t}{b}-\frac{\zeta}{b} ) }{s_b\big(\frac{t}{b}+\frac{i Q}{2}\big)} \frac{1}{b}
= \frac{e^{\frac{\pi i}{4}}}{b} e^{\frac{f(t)}{b^2}} g(t) e^{E(b^2,t) + O(b^6)} \qquad \text{as $b \to 0^+$}
\end{align}
uniformly for $t \in b\mathsf{M}$, where the functions $f(t) = f(t; \zeta, \omega)$ and $g(t)$ are analytic function of $t$ in the strip $-1/2 < \im t < 0$ defined by
\begin{align}\nonumber
& f(t) = - \frac{\Li_{2}(-e^{2\pi \zeta})+\Li_{2}(-e^{2\pi (t-\zeta)})-\Li_{2}(-e^{2\pi (t + \frac{i}{2})})}{2 \pi i}
+  i \pi  \zeta^2 +2 i \pi  t \omega +\pi  t+\frac{i \pi }{6},
	\\ \label{Mgdef}
& g(t) = e^{\pi  t-\frac{1}{2} \ln \left(1-e^{2 \pi  t}\right)},
\end{align}
and the function $E(\tau,t) = E(\tau,t; \zeta, \omega)$ is given by
\begin{align*}
E(\tau,t) =&\; \frac{i \big(\text{Li}_2(e^{2 \pi  t})-\text{Li}_2(e^{i \pi  \tau +2 \pi  t})\big)}{2 \pi  \tau} +\frac{\ln(1-e^{2 \pi  t})}{2}
   	\\
&  +\frac{i \pi  \tau}{12}  \bigg(\frac{e^{2 \pi  t+i \pi  \tau}}{e^{2 \pi  t+i \pi  \tau} - 1}-\frac{e^{2 \pi  \zeta }}{e^{2 \pi  \zeta}+1}-\frac{e^{2 \pi  t}}{e^{2 \pi  \zeta }+e^{2 \pi  t}}+2\bigg).
   \end{align*}   
The next lemma shows that $E(b^2, t)$ is small as $b \to 0^+$ uniformly for $t \in b\mathsf{M}$. 
   
\begin{lemma}\label{Elemma}
There is a constant $C > 0$ such that
$$|E(b^2,t)| \leq C b^2$$
for all sufficiently small $b > 0$ and all $t$ in the strip $\im t \in (-1/2 + \delta, -2\delta)$.
\end{lemma}
\begin{proof}
For any $t \in \C$ with $\im t \in (-1/2 + \delta, -2\delta)$, the function $E(b^2,t)$ has a continuous extension to $b = 0$ such that $E(0,t) = 0$. Thus $E(b^2,t) = \int_0^{b^2} \partial_\tau E(\tau, t) d\tau$, implying that the desired conclusion will follow if we can show that $|\partial_\tau E(\tau, t)|$ is uniformly bounded for all sufficiently small $\tau > 0$ and all $t$ in the strip $\im t \in (-1/2 + \delta, -2\delta)$.

We have
\begin{align}\label{partialtauE}
\partial_\tau E(\tau, t) = \frac{E_1(\tau, t)}{2\pi i \tau^2}  + \frac{\pi i E_2(\tau, t)}{12},
\end{align}
where
\begin{align*}  
& E_1(\tau, t) := -\text{Li}_2\left(e^{2 \pi  t+i \pi  \tau }\right)+\text{Li}_2\left(e^{2 \pi t}\right)-i \pi  \tau  \ln \left(1-e^{2 \pi  t+i \pi  \tau }\right),
   	\\
& E_2(\tau, t) := 2-\frac{1}{e^{-2 \pi  \zeta }+1}-\frac{1}{e^{2 \pi  (\zeta- t) }+1}+\frac{e^{2 \pi  t+i \pi  \tau } \left(e^{2 \pi  t+i \pi  \tau }-i \pi \tau -1\right)}{\left(e^{2 \pi  t+i \pi  \tau } - 1\right)^2}.
\end{align*}
It is easy to see that the function $E_2(\tau, t)$ remains bounded for all $t$ in the strip $\im t \in (-1/2 + \delta, -2\delta)$ whenever $\tau > 0$ is small enough.
Also, the definition (\ref{dilogdef}) of $\Li_2$ gives
\begin{align*}  
& E_1(\tau, t) = \int_{e^{2 \pi  t}}^{e^{2 \pi  t+i \pi  \tau }} \frac{\ln(1-u)}{u} du  -i \pi  \tau  \ln\left(1-e^{2 \pi  t+i \pi  \tau }\right).
\end{align*}
We see that for any $t \in \C$ with $\im t \in (-1/2 + \delta, -2\delta)$, the function $E_1(\tau,t)$ has a continuous extension to $\tau = 0$ such that $E_1(0,t) = 0$. Thus 
$$E_1(\tau,t) = \int_0^{\tau} \partial_{\tau'} E_1(\tau', t) d\tau' 
= \int_0^{\tau} \frac{\pi ^2 \tau'}{1-e^{-2 \pi  t-i \pi  \tau'}} d\tau',$$
and so
$$|E_1(\tau,t)| \leq C \int_0^{\tau} \tau' d\tau' = \frac{C}{2} \tau^2,$$
for all sufficiently small $\tau > 0$ and all $t$ in the strip $\im t \in (-1/2 + \delta, -2\delta)$.
In light of (\ref{partialtauE}), this shows that $|\partial_\tau E(\tau, t)|$ is uniformly bounded for all small $\tau > 0$ and all $t$ in the strip $\im t \in (-1/2 + \delta, -2\delta)$.
\end{proof}

Substituting (\ref{calMintegrandexpansion}) into (\ref{calMtintegral}) and using Lemma \ref{Elemma}, we conclude that
\begin{align}\label{calMfg}
\mathcal{M}\bigg(b, \frac{\zeta}{b}, \frac{\omega}{b}\bigg) 
=&\; \frac{e^{\frac{\pi i}{4}}}{b} \int_{b\mathsf{M}} e^{\frac{f(t)}{b^2}} g(t) dt \big(1 + O(b^2)\big) \qquad \text{as $b \to 0^+$}.
\end{align}

The next lemma shows that the integrand in (\ref{calMfg}) has exponential decay as $t \in b\mathsf{M}$ tends to infinity.

\begin{lemma}\label{Mefgatinftylemma}
Let $t = t_1 + it_2$. There are constants $C > 0$ and $c>0$ such that 
\begin{align}\label{Mefgatinfty}
& \big|e^{\frac{f(t)}{b^2}} g(t)\big| \leq C e^{-\frac{c}{b^2}|t_1|}
\end{align}
for all sufficiently large $|t_1|$, for all $t_2 \in (-1/2, 0)$, and for all sufficiently small $b> 0$. 
\end{lemma}
\begin{proof}
From (\ref{dilogdef}), we see that
\begin{align}\label{Li2asymptotics}
\Li_2(-z) = -\frac{1}{2}(\ln z)^2 + O(1) \qquad \text{as $z \to \infty$}
\end{align}
uniformly for $\arg(z) \in (-\pi, \pi)$. Hence, using that $\im \Li_{2}(-e^{2\pi \zeta}) = 0$, we obtain 
\begin{align*}
\re f(t)  
& = -\im\left(\frac{-\frac{1}{2}(2\pi (t-\zeta))^2 + \frac{1}{2}(2\pi (t + \frac{i}{2}))^2}{2 \pi }\right) 
- 2\pi t_1 \im \omega +\pi  t_1 + O(1)
	\\
& =  - 2\pi t_1 \im \omega + O(1)
\end{align*}
as $t_1 \to +\infty$ uniformly for $t_2 \in (-1/2, 0)$. Since $\im \omega > 0$ and $|g(t)| = O(1)$ as $t_1 \to +\infty$ uniformly for $t_2 \in (-1/2, 0)$, the estimate for large positive values of $t_1$ follows.

As $t_1 \to -\infty$, we have $\re f(t) = (2\pi \im \omega -\pi ) |t_1| + O(1)$ and hence there is a $c > 0$ such that $\re f(t) < -c|t_1|$ for all sufficiently large negative $t_1$ and all $t_2 \in (-1/2, 0)$.
Since $|g(t)| = O(1)$ as $t_1 \to -\infty$ uniformly for $t_2 \in (-1/2, 0)$, the estimate for large negative values of $t_1$ also follows.
\end{proof}

The first and second derivatives of $f$ are given by
\begin{align}\label{Mdfdt}
f'(t) & =-i \ln \left(1 + e^{2 \pi  (t-\zeta )}\right)+i \ln \left(1- e^{2 \pi t}\right)+2 i \pi  \omega +\pi,
	\\ \label{Md2fdt2}
f''(t) & = \frac{\pi i  \cosh (\pi  \zeta ) }{\sinh(\pi  t) \cosh(\pi 
   (t-\zeta ))}.
\end{align}
In particular, the saddle-point equation $f'(t) = 0$ takes the form
\begin{align}\label{Msaddlepointeq}
\ln \left(1 + e^{2 \pi  (t-\zeta )}\right) - \ln \left(1- e^{2 \pi t}\right) = 2 \pi  \omega - i \pi.
\end{align}

\begin{lemma}\label{Mt0lemma}
The saddle point equation (\ref{Msaddlepointeq}) has a unique solution $t_0 = t_0(\zeta, \omega)$ in the strip $-1/2 < \im t < 0$ given by
\begin{align}\label{Mt0def}
t_0 = \frac{1}{2\pi} \ln\Big(e^{\pi  \zeta } \frac{\cosh (\pi  \omega )}{\sinh(\pi  (\zeta +\omega))}\Big).
\end{align}
\end{lemma}
\begin{proof}
Since
$$\im \bigg(e^{\pi  \zeta } \frac{\cosh (\pi  \omega )}{\sinh(\pi  (\zeta +\omega))}\bigg)
= - \frac{e^{\pi  \zeta } \cosh (\pi  \zeta ) \sin (2 \pi  \im{\omega})}{\cosh (2 \pi  (\zeta + \re{\omega}))- \cos (2 \pi  \im{\omega})} < 0,$$
it follows that $-1/2 < \im t_0 < 0$. 
On the other hand, letting $z = e^{2\pi t}$, the saddle-point equation (\ref{Msaddlepointeq}) is
\begin{align}
\ln(1 + z e^{-2\pi \zeta}) - \ln(1 - z) = 2 \pi  \omega - \pi i.
\end{align}
Exponentiation gives the equation
$$\frac{1 + z e^{-2\pi \zeta}}{1 -z} = -e^{2 \pi  \omega},$$
which has the unique solution
$$z_0 = e^{\pi  \zeta } \frac{\cosh (\pi  \omega )}{\sinh(\pi  (\zeta +\omega))}.$$
It follows that any solution of (\ref{Msaddlepointeq}) has the form $t_0 + in$ for some integer $n$; in particular, there can be at most one solution in the strip $-1/2 < \im t < 0$. It also follows that $t_0$ fulfills the saddle point equation (\ref{Msaddlepointeq}) modulo an integer multiple of $2\pi i$, i.e.,
\begin{align}\label{Msaddlemod}
\ln \left(1 + e^{2 \pi  (t_0 -\zeta )}\right) - \ln \left(1- e^{2 \pi t_0}\right) - 2 \pi  \omega + i \pi = 2\pi i N,
\end{align}
for some $N \in \Z$. Since $\im \omega \in (0,1/2)$,
\begin{align}\label{Mtsaddleimlog}
\im \ln \left(1 + e^{2 \pi  (t_0-\zeta )}\right) \in (-\pi,0), \quad \text{and} \quad
\im \ln \left(1- e^{2 \pi t_0}\right) \in (0,\pi),
\end{align}
we conclude that the imaginary part of the left-hand side of (\ref{Msaddlemod}) lies in the open interval $(-2\pi,\pi)$, so the only possibility is $N = 0$. This shows that $t_0$ satisfies (\ref{Msaddlepointeq}) and concludes the proof.
\end{proof}

We henceforth let $t_0$ be the solution (\ref{Mt0def}) of the saddle point equation. 

\begin{lemma}\label{Mrefprimeprimelemma}
We have
\begin{align}\label{Mrefppnegative}
\re f''(t) < 0 \qquad \text{for all $t$ in the strip $-1/2 < \im t < 0$.}
\end{align}
In particular, $\arg\big(-f''(t_0)\big) \in (-\pi/2, \pi/2)$.
\end{lemma}
\begin{proof}
Taking the real part of (\ref{Md2fdt2}), we find, for all $t = t_1 + it_2$ with $-1/2 < t_2 < 0$,
$$\re f''(t) =
\frac{2 \pi  \cosh (\pi  \zeta ) \sin (2 \pi  t_2) \cosh (\pi (2 t_1-\zeta ))}{(\cosh (2 \pi  t_1) - \cos (2 \pi  t_2)) (\cosh (2 \pi  (t_1-\zeta )) + \cos (2 \pi  t_2))} < 0,$$
where the last inequality follows because $\sin(2\pi t_2) < 0$. 
\end{proof}

We define $\alpha_0$ by
\begin{align}\label{Malpha0def}
\alpha_0 = -\frac{\arg(-f''(t_0))}{2}.
\end{align}
In view of Lemma \ref{Mrefprimeprimelemma}, we have $f''(t_0) \neq 0$ and $\alpha_0 \in (-\pi/4,\pi/4)$.
Since $f''(t_0) \neq 0$, the implicit function theorem implies that $\{t \,|\, \im f(t) = \im f(t_0)\}$ consists of two smooth curves in a neighborhood of $t_0$ that cross at right angles at $t_0$: one of these curves is the steepest descent contour with tangent line $t_0 + \R e^{i\alpha_0}$ at $t_0$, and the other one is the steepest ascent contour with tangent line $t_0 + i \R e^{i\alpha_0}$ at $t_0$. We define the function $\varphi$ for $t$ near $t_0$ by
\begin{align}\label{varphidef}
\varphi(t) = |f''(t_0)|^{1/2} e^{-i\alpha_0} (t-t_0)\sqrt{\frac{2(f(t) - f(t_0))}{f''(t_0)(t-t_0)^2}},
\end{align}
where the principal branch is used for the square roots. The map $\varphi$ is an analytic bijection from a neighborhood $U$ of $t_0$ to a neighborhood of $0$ such that
\begin{align}\label{fvarphi}
f(t) - f(t_0) = -\frac{\varphi(t)^2}{2}
\end{align}
and such that $\varphi$ maps $\{t \,|\, \im f(t) = \im f(t_0)\} \cap U$ into $\R \cup i\R$. The steepest descent (resp. ascent) contour restricted to $U$ is mapped into $\R$ (resp. $i\R$). 
Let $D_\epsilon(z)$ denote the open disk of radius $\epsilon$ centered at $z$. 
For all small enough $\epsilon > 0$, the open set $\mathcal{D}_\epsilon := \varphi^{-1}(D_\epsilon(0))$ is contained in $U$. Shrinking $\epsilon > 0$ if necessary, we may assume that $\re \varphi^{-1}(x)$ is an increasing function of $x \in [-\epsilon, \epsilon]$, and that the steepest descent contour intersects $\partial \mathcal{D}_\epsilon$ in the two points $\varphi^{-1}(-\epsilon)$ and $\varphi^{-1}(\epsilon)$.
We define the contour $\Gamma = \Gamma_1 \cup \Gamma_2 \cup \Gamma_3$ as the union of the horizontal rays 
$$\Gamma_1 = \varphi^{-1}(-\epsilon) + \R_{\leq 0} \quad \text{and} \quad
\Gamma_3 = \varphi^{-1}(\epsilon) + \R_{\geq 0},$$
and the steepest descent curve $\Gamma_2 = \varphi^{-1}([-\epsilon, \epsilon])$, see Figure \ref{steepest.pdf}.

\begin{figure}
\bigskip\begin{center}
\hspace{-.4cm}
\begin{overpic}[width=0.9\textwidth]{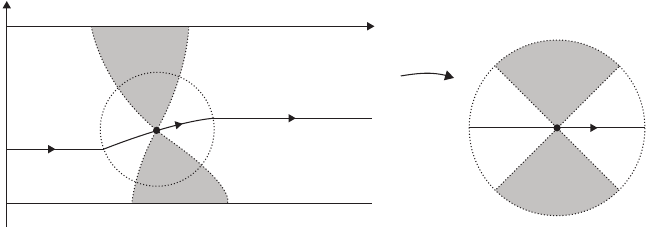}
      \put(58.5,30.5){\footnotesize $\re t $}
      \put(-1,36){\footnotesize $\im t$}
      \put(-.9,30.5){\footnotesize $0$}
      \put(-3,3.2){\footnotesize $-\frac{1}{2}$}
      \put(7,9.5){\footnotesize $\Gamma_1$}
      \put(28,13.9){\footnotesize $\Gamma_2$}
      \put(44.2,14.3){\footnotesize $\Gamma_3$}
      \put(23.8,12.4){\footnotesize $t_0$}
      \put(31,22){\footnotesize $\partial \mathcal{D}_\epsilon$}
      \put(99,22){\footnotesize $\partial D_\epsilon(0)$}
      \put(65,25){\footnotesize $\varphi$}
      \put(85.5,13){\footnotesize $0$}
      \put(100.2,14){\footnotesize $\epsilon$}
      \put(69,14){\footnotesize $-\epsilon$}
    \end{overpic}
     \begin{figuretext}\label{steepest.pdf}
       Schematic illustration of the contour $\Gamma = \Gamma_1 \cup \Gamma_2 \cup \Gamma_3$ in the complex $t$-plane, as well as the map $\varphi$ from $\mathcal{D}_\epsilon$ to $D_\epsilon(0)$. The regions where $\re(f(t) - f(t_0)) > 0$ are shaded. The circle $\partial D_\epsilon(0)$ and its inverse image $\partial \mathcal{D}_\epsilon = \varphi^{-1}(\partial D_\epsilon(0))$ are dotted. The map $\varphi$ maps the steepest descent curve $\Gamma_2$ to the interval $[-\epsilon, \epsilon]$. 
       \end{figuretext}
     \end{center}
\end{figure}

Deforming the contour to $\Gamma$, we may write (\ref{calMfg}) as
\begin{align}\label{calMI1I2I3}
\mathcal{M}\bigg(b, \frac{\zeta}{b}, \frac{\omega}{b}\bigg) 
=&\; \frac{e^{\frac{\pi i}{4}}}{b} e^{\frac{f(t_0)}{b^2}} \big(I_1(b) + I_2(b) + I_3(b)\big) \big(1 + O(b^2)\big),
\end{align}
where the functions $I_j(b)$ are defined by
$$I_j(b) = \int_{\Gamma_j} e^{\frac{f(t) - f(t_0)}{b^2}} g(t) dt, \qquad j = 1,2,3.$$

Our next goal is to show that $I_1(b)$ and $I_3(b)$ are exponentially small in the limit $b \to 0$. To this end, we need the following lemma.

\begin{lemma}\label{Mreflemma}
 Shrinking $\epsilon > 0$ if necessary, the following estimate holds for all $t \in \Gamma_1 \cup \Gamma_3$ for some $c_0 > 0$:
$$\re \big[f(t) - f(t_0)\big] < -c_0.$$
\end{lemma}
\begin{proof}
By the definition of $\varphi$, we have
$$\re \big[f(\varphi^{-1}(\pm \epsilon)) - f(t_0)\big] < 0$$
for all sufficiently small $\epsilon > 0$. So the desired estimate holds at the finite end points of $\Gamma_1$ and $\Gamma_3$. The lemma will follow if we can show that
\begin{align}\label{Mrefestimate}
\re f(t) \leq \begin{cases} \re f(\varphi^{-1}(-\epsilon)) & \text{for all $t \in \Gamma_1$},
	\\
\re f(\varphi^{-1}(\epsilon)) & \text{for all $t \in \Gamma_3$},
\end{cases}
\end{align}
for all sufficiently small $\epsilon > 0$. 
By (\ref{Mrefppnegative}), we have $\re f''(t) < 0$ everywhere in the strip $-1/2 < \im t < 0$, so for each $t_2 \in (-1/2, 0)$, $\re f(t_1 + it_2)$ is a strictly concave function of $t_1 \in \R$.
Also, for each $t_2 \in (-1/2,0)$, there is a unique solution $t_1$ of the equation $\re f'(t_1 + it_2) = 0$; we denote the solution corresponding to $t_2$ by $\hat{t}_1(t_2)$. By the implicit function theorem, $\hat{t}_1(t_2)$ defines a smooth curve in a neighborhood of $t_2 = \im t_0$ that passes through $t_0$ and whose tangent line at $t_0$ is $t_0 + \R e^{i(2\alpha_0 - \frac{\pi}{2})}$. By concavity, $\re f'(t)  > 0$ if $t$ lies to the left of this curve (i.e., if $t_1 < \hat{t}_1(t_2)$) and $\re f'(t) < 0$ if $t$ lies to the right of this curve (i.e., if $t_1 > \hat{t}_1(t_2)$). Since $\alpha_0 - ( 2\alpha_0 - \frac{\pi}{2}) \in (\pi/4, 3\pi/4)$, we conclude that $\varphi^{-1}(\epsilon)$ lies to the right of the curve, while $\varphi^{-1}(-\epsilon)$ lies to the left of the curve for all sufficiently small $\epsilon > 0$. The estimate (\ref{Mrefestimate}) follows, which completes the proof.
\end{proof}

We are now ready to show that $I_1(b)$ and $I_3(b)$ are exponentially small as $b \to 0^+$.

\begin{lemma}\label{MI1I3lemma}
There exists a $c > 0$ such that
$$I_1(b) = O(e^{-\frac{c}{b^2}}) \quad \text{and} \quad I_3(b) = O(e^{-\frac{c}{b^2}}) \qquad \text{as $b \to 0^+$}.$$
\end{lemma}
\begin{proof}
By Lemma \ref{Mefgatinftylemma}, there are constants $M > 0$, $C > 0$, and $c> 0$ such that, for all sufficiently small $b > 0$,
$$\bigg|\int_{\Gamma \cap \{|\re t| > M\}} e^{\frac{f(t) - f(t_0)}{b^2}} g(t) dt\bigg|
\leq C e^{-\frac{\re f(t_0)}{b^2}} \int_M^{+\infty} e^{-\frac{c}{b^2}t_1} dt_1 
\leq C e^{-\frac{\re f(t_0)}{b^2}}  e^{-\frac{c M}{b^2}}.
$$
By choosing $M$ so large that $c M > |\re f(t_0)|$, we see that the contribution from $\Gamma \cap \{|\re t| > M\}$ is of order $O(e^{-\frac{c}{b^2}})$ for some $c > 0$.
On the other hand, we know from Lemma \ref{Mreflemma} that $\re \big[f(t) - f(t_0)\big] < -c_0$ for $t \in \Gamma_1 \cup \Gamma_3$, and hence
$$\bigg|\int_{(\Gamma_1 \cup \Gamma_3) \cap \{|\re t| \leq M\}} e^{\frac{f(t) - f(t_0)}{b^2}} g(t) dt\bigg|
\leq 2M e^{-\frac{c_0}{b^2}} \sup_{t \in \Gamma_1 \cup \Gamma_3} |g(t)| \leq C e^{-\frac{c_0}{b^2}},$$
so the desired conclusion follows.
\end{proof}

The next lemma determines the asymptotics of $I_2(b)$.

\begin{lemma}\label{MI2lemma}
As $b \to 0^+$, it holds that
$$I_2(b) = \frac{g(t_0) \sqrt{2\pi}  }{ |f''(t_0)|^{1/2} e^{-i\alpha_0} }b
+ O(b^2).$$
\end{lemma}
\begin{proof}
Performing the change of variables $u = \varphi(t)$, we obtain
\begin{align*}
I_2(b) & = \int_{-\epsilon}^{\epsilon} e^{-\frac{u^2}{2 b^2}} g(\varphi^{-1}(u)) \frac{du}{\varphi'(\varphi^{-1}(u))}
 = \frac{g(t_0)}{\varphi'(t_0)}  \int_{-\epsilon}^{\epsilon} e^{-\frac{u^2}{2 b^2}} du
+ O\bigg(\int_{-\epsilon}^{\epsilon} e^{-\frac{u^2}{2 b^2}} |u| du\bigg)
	\\
& = \frac{g(t_0)}{ |f''(t_0)|^{1/2} e^{-i\alpha_0} } \int_{-\epsilon}^{\epsilon} e^{-\frac{u^2}{2 b^2}} du
+ O(b^2)
 = \frac{g(t_0) \sqrt{2\pi}  }{ |f''(t_0)|^{1/2} e^{-i\alpha_0} }b
+ O(b^2) \qquad \text{as $b \to 0^+$},
\end{align*}
which is the desired conclusion.
\end{proof}

Employing Lemma \ref{MI1I3lemma} and Lemma \ref{MI2lemma} in (\ref{calMI1I2I3}), we arrive at
\begin{align}\label{calMasymptotics}
\mathcal{M}\bigg(b, \frac{\zeta}{b}, \frac{\omega}{b}\bigg) 
=&\; \frac{e^{\frac{\pi i}{4}} g(t_0) \sqrt{2\pi}  }{ |f''(t_0)|^{1/2} e^{-i\alpha_0} } e^{\frac{f(t_0)}{b^2}} \big(1 + O(b)\big) \qquad \text{as $b \to 0^+$}.
\end{align}
It is possible to simplify the coefficient in (\ref{calMasymptotics}) as the next lemma shows. 

\begin{lemma}\label{Megfsimplifiedlemma}
We have
\begin{align}\label{Megfsimplified}
\frac{e^{\frac{\pi i}{4}} g(t_0) \sqrt{2\pi}  }{ |f''(t_0)|^{1/2} e^{-i\alpha_0} } = \sqrt{\frac{\coth (\pi  (\zeta +\omega ))+1}{2 }},
\end{align}
where the principal branch is used for the square root. 
\end{lemma}
\begin{proof}
At the saddle point $t_0$, a direct calculation shows that
\begin{align*}
& f''(t_0) = 4 i \pi  \frac{\cosh (\pi  \omega )}{\cosh(\pi  \zeta ) } \sinh (\pi   (\zeta +\omega )).
\end{align*}
Hence, recalling the definition (\ref{Malpha0def}) of $\alpha_0$ and using that $\arg\big(-f''(t_0)\big) \in (-\pi/2, \pi/2)$ by Lemma \ref{Mrefprimeprimelemma}, we deduce that
$$|f''(t_0)|^{1/2} e^{-i\alpha_0} = \sqrt{-f''(t_0)} = \sqrt{-4 i \pi  \frac{\cosh (\pi  \omega )}{\cosh(\pi  \zeta ) } \sinh (\pi   (\zeta +\omega ))},$$
where the principal branch is used for the square roots. 
On the other hand, substituting the expression (\ref{Mt0def}) for $t_0$ into the definition (\ref{Mgdef}) of $g$, we obtain 
$$g(t_0) = \frac{\sqrt{e^{\pi  \zeta } \cosh (\pi  \omega ) \text{csch}(\pi   (\zeta +\omega ))}}{\sqrt{1-e^{\pi  \zeta } \cosh (\pi  \omega ) \text{csch}(\pi  (\zeta +\omega ))}},$$
where the principal branch is used for the square roots. It follows that
$$\bigg(\frac{e^{\frac{\pi i}{4}} g(t_0) \sqrt{2\pi}  }{ |f''(t_0)|^{1/2} e^{-i\alpha_0} }\bigg)^2 = \frac{\coth (\pi  (\zeta +\omega ))+1}{2},$$
showing that (\ref{Megfsimplified}) holds at least up to sign. To fix the sign, note that the definition (\ref{Mgdef}) of $g$ implies that
$$\arg g(t) = \pi t_2 - \frac{1}{2} \arg (-i(1-e^{2 \pi  t})) - \frac{\pi}{4}
= \pi t_2 - \frac{1}{2} \arctan\Big(\frac{-e^{-2\pi t_1} + \cos(2\pi t_2)}{-\sin(2\pi t_2)}\Big) - \frac{\pi}{4}$$
Hence, for $t = t_1 + it_2$ with $-1/2 < t_2 < 0$, we have
$$\frac{d \arg g(t)}{dt_1} = \frac{\pi  \sin (2 \pi  t_2)}{2 (\cosh
   (2 \pi  t_1) - \cos (2 \pi  t_2))} < 0 \qquad \text{for all $t_1 \in \R$}.$$
Since $\lim_{t_1 \to -\infty} \arg g(t) = \pi t_2 \in (-\pi/2, 0)$ and $\lim_{t_1 \to +\infty} \arg g(t) = -\pi/2$, we conclude that $\arg g(t) \in (-\pi/2, 0)$ for all $t$ in the strip $-1/2 < \im t < 0$; in particular, $\arg g(t_0) \in (-\pi/2, 0)$. Since $\alpha_0 \in (-\pi/4, \pi/4)$, it follows that
$$\arg \frac{e^{\frac{\pi i}{4}} g(t_0) \sqrt{2\pi}  }{ |f''(t_0)|^{1/2} e^{-i\alpha_0} } \in \Big(-\frac{\pi}{2}, \frac{\pi}{2}\Big),$$
showing that (\ref{Megfsimplified}) holds when the principal branch is used for the square root.  
\end{proof}

The function $f_{\mathcal{M}}(\zeta, \omega)$ defined by $f_{\mathcal{M}}(\zeta, \omega) := f(t_0; \zeta, \omega)$ can be written as in (\ref{finthmM}). Consequently, the asymptotic formula (\ref{Mfinalasymptotics}) follows from (\ref{calMasymptotics}) and Lemma \ref{Megfsimplifiedlemma}. This completes the proof of \eqref{Mfinalasymptotics}.

\subsection{$f_\mathcal{M}$ and the Painlev\'e I generating function}
We now turn to the proof of the second and last statement of Theorem \ref{Mth} which relates $f_\mathcal{M}$ to the Painlev\'e I generating function of Definition \ref{defgeneratingPI}. The key to the proof is the following.

\begin{proposition} \label{prop:diffeqforf}
For any $\zeta \in \R$ and any $\omega \in \C$ such that $0 < \im \omega < 1/2$, the function $f_\mathcal{M}$ defined in \eqref{finthmM} satisfies the following differential equations:
\begin{align}
\label{diffeqfzeta} & \frac{\partial f_{\mathcal{M}}(\zeta,\omega)}{\partial \zeta} = 2 i \pi (\zeta + \omega) +\pi - i \ln\lb 1-e^{2\pi(\zeta+\omega)}\rb, \\
\label{diffeqfomega} & \frac{\partial f_{\mathcal{M}}(\zeta,\omega)}{\partial \omega} = \pi + 2i\pi \zeta -i \ln \lb \frac{1-e^{2\pi(\zeta+\omega)}}{1+e^{2\pi \omega}} \rb.
\end{align}
\end{proposition}
\begin{proof}
The proofs of \eqref{diffeqfzeta} and \eqref{diffeqfomega} readily follow by differentiating \eqref{finthmM} utilizing the derivative formula
\begin{equation}
\frac{d \Li_2(z)}{dz} = - \frac{\ln(1-z)}{z},
\end{equation}
together with the fact that $\zeta \in \R$ and $0 < \im \omega < 1/2$.
\end{proof}

If $\zeta \in -\frac{1}{2} + i\R$ and $\omega$ belongs to the strip $\re \omega \in (1/2, 1)$, then $i\zeta + \frac{i}{2} \in \R$ and $i\omega-\frac{i}{2}$ belongs to the strip $0 < \im \omega < 1/2$. Using Proposition \ref{prop:diffeqforf}, it is therefore a routine calculation to show that the function $Y$ defined in \eqref{identificationPIgenfunc} satisfies \eqref{diffeqWzeta} and \eqref{diffeqWomega} for $\zeta \in -\frac{1}{2} + i\R$ and $\omega$ in the strip $\re \omega \in (1/2, 1)$; by analyticity, the relations \eqref{diffeqWzeta} and \eqref{diffeqWomega} extend to other values of $(\zeta, \omega)$. This completes the proof of Theorem \ref{Mth}.

\section{Semiclassical limit of $\mathcal Q$: Proof of Theorem \ref{Qth}} \label{section5}

\subsection{Proof of the asymptotic formula (\ref{Qfinalasymptotics}) for $\mathcal{Q}$}
Fix $(\sigma_s,\mu) \in \mathbb R^2$. Making the change of variables $x = \frac{t}{b}$ in the definition \eqref{defQ} of $\mathcal Q$, we obtain
\begin{align}\label{calQtintegral}
\mathcal{Q}\bigg(b, \frac{\sigma_s}{b}, \frac{\mu}{b}\bigg) 
= \int_{b\mathsf{Q}} \mathcal{I}(b, \sigma_s, \mu) dt,
\end{align}
where
$$\mathcal{I}(b, \sigma_s, \mu) := e^{i \pi  \big(\frac{1}{6}+\frac{7Q^2}{24} -\frac{\mu^2}{2 b^2}+i \frac{\mu}{b} Q - \frac{\sigma_s^2}{b^2}\big)} s_b\bigg(\frac{\mu}{b}\bigg) e^{-\frac{i\pi t^2}{b^2}} e^{2i\pi \frac{t}{b}(\frac{\mu}{b}-\frac{iQ}{2})} s_b\bigg(\frac{t + \sigma_s}{b}\bigg) s_b\bigg(\frac{t-\sigma_s}{b}\bigg) \frac{1}{b}.$$
The integration contour $b\mathsf{Q}$ can be any curve from $-\infty$ to $+\infty$ passing above the points $t = \pm \sigma_s - \frac{i}{2} - \frac{i}{2}b^2$ such that its right tail satisfies
\beq
\im t + \tfrac{1}{4} + \tfrac{b^2}{4} - \tfrac{\im \mu}{2} < -b \delta \qquad
\text{for all $t \in b\mathsf{Q}$ with $\re t$ sufficiently large},
\eeq
for some $\delta>0$.  Since $(\sigma_s,\mu) \in \mathbb R^2$, we may choose the contour $b\mathsf{Q}$ to be any curve from $-\infty$ to $+\infty$ lying within the strip $\im t \in (-\frac{1}{2} - \frac{b^2}{2}, -\frac{1}{4} - \frac{b^2}{4} - b \delta)$. We choose $b\mathsf{Q}$ to be a contour independent of $b$ that lies in the strip $\im t \in (-\frac{1}{2}+\delta, -\frac{1}{4}-\delta)$. Then $\mu$ and $t \pm \sigma_s$ lie in $S_\delta$ for all $t \in b\mathsf{Q}$ and all small enough $b>0$, where $S_\delta$ is the strip defined in (\ref{Sdeltadef}).
Therefore we can apply the expansion for $s_b(z/b)$ obtained in Proposition \ref{sbprop} with $N = 1$ to 
deduce that the integrand in (\ref{calQtintegral}) satisfies 
\begin{align}\label{calQintegrandexpansion}
\mathcal{I}(b, \sigma_s, \mu)
= \frac{e^{\frac{3 \pi i}{4}}}{b} e^{\frac{f(t)}{b^2}} g(t) e^{b^2 H(t) + O(b^6)} \qquad \text{as $b \to 0^+$}
\end{align}
uniformly for $t \in b\mathsf{M}$, where the functions $f(t) = f(t; \sigma_s, \mu)$ and $g(t) = g(t; \mu)$ are analytic function of $t$ in the strip $-1/2 < \im t < 1/2$ defined by\footnote{These functions should not be confused with the functions in (\ref{Mgdef}). To stress the analogy with the proof  in Section \ref{section4}, we use the same notation here as we did in Section \ref{section4} for the functions $f$ and $g$ as well as some other quantities. }
\begin{align}\nonumber
& f(t) = -\frac{\Li_2(-e^{2 \pi  \mu }) +  \Li_2(-e^{2 \pi  (t-\sigma_s)}) + \Li_2(-e^{2 \pi(t+\sigma_s)})}{2 \pi i}+2 i \pi  \mu  t+\pi (t - \mu) +\frac{5 i \pi }{12},
	\\ \label{Qgdef}
& g(t) = e^{\pi (t- \mu)},
\end{align}
and the function $H(t) = H(t; \sigma_s, \mu)$ is given by
\begin{align*}
H(t) =&\; \frac{i \pi}{12 \left(e^{2 \pi  \mu }+1\right) \left(e^{2 \pi \sigma_s}+e^{2 \pi  t}\right) \left(e^{2 \pi  (\sigma_s+t)}+1\right)}
 \Big(4 e^{2 \pi  (\mu +\sigma_s)}+5 e^{2 \pi  \sigma_s}
+2 e^{2 \pi  (\mu +\sigma_s+2 t)}
	\\
& +3 e^{2 \pi  (\mu +2 \sigma_s+t)}+3 e^{2 \pi  (\mu +t)}+3 e^{2 \pi  (\sigma_s+2 t)}+4 e^{2 \pi  (2 \sigma_s+t)}+4 e^{2 \pi  t}\Big).
   \end{align*}   
It is easily seen that $H(t)$ is bounded for all $t \in S_\delta$. Hence, substitution of (\ref{calQintegrandexpansion}) into (\ref{calQtintegral}) gives
\begin{align}\label{calQfg}
\mathcal{Q}\bigg(b, \frac{\sigma_s}{b}, \frac{\mu}{b}\bigg) 
& = \frac{e^{\frac{3\pi i}{4}}}{b} \int_{b\mathsf{Q}}  e^{\frac{f(t)}{b^2}} g(t) dt \big(1 + O(b^2)\big).
\end{align}

The next lemma will be useful to estimate the tails at $\pm \infty$ of the integral in (\ref{calQfg}).

\begin{lemma}\label{Qefgatinftylemma}
Let $t = t_1 + it_2$. There are constants $C > 0$ and $c>0$ such that 
\begin{align}\label{Qefgatinfty}
& \big|e^{\frac{f(t)}{b^2}} g(t)\big| \leq C e^{-\frac{c}{b^2}|t_1|}
\end{align}
for all sufficiently large $|t_1|$, for all $t_2 \in (-1/2, -1/4-\delta)$, and for all sufficiently small $b> 0$. 
\end{lemma}
\begin{proof}
Using (\ref{Li2asymptotics}) and the fact that $\im \Li_{2}(-e^{2\pi \mu}) = 0$, we obtain 
\begin{align*}
\re f(t)  
& = - \im \frac{ \Li_2\left(-e^{2 \pi  (t-\sigma_s)}\right) + \Li_2\left(-e^{2 \pi(t+\sigma_s)}\right)}{2 \pi }- 2 \pi  \mu  t_2+\pi (t_1 - \mu) 
	\\
& = -\im\left(\frac{-\frac{1}{2}(2\pi (t-\sigma_s))^2 - \frac{1}{2}(2\pi (t + \sigma_s))^2}{2 \pi }\right) 
+ \pi  t_1 + O(1)
	\\
& =  \pi t_1 (1 + 4t_2) + O(1)
\end{align*}
as $t_1 \to +\infty$ uniformly for $t_2 \in (-1/2, -1/4)$. 
Hence there is a $c > 0$ such that $\re f(t) < -c|t_1|$ for all sufficiently large positive $t_1$ and all $t_2 \in (-1/2, -1/4-\delta)$.
Since $|g(t)| = O(e^{\pi t_1})$ as $t_1 \to +\infty$ uniformly for $t_2 \in (-1/2, 0)$, the estimate for large positive values of $t_1$ follows.

As $t_1 \to -\infty$, we have $\re f(t) = - \pi |t_1| + O(1)$ uniformly for $t_2 \in (-1/2, -1/4)$.
Since $|g(t)| = O(1)$ as $t_1 \to -\infty$ uniformly for $t_2 \in (-1/2, -1/4)$, the estimate for large negative values of $t_1$ also follows.
\end{proof}

The first and second derivatives of $f$ are given by
\begin{align}
f'(t) & = -i \ln \left(e^{2 \pi  (t-\sigma_s)}+1\right)-i \ln \left(e^{2 \pi  (\sigma_s+t)}+1\right)
+2 i \pi  \mu+\pi,
	\\ \label{Qd2fdt2}
f''(t) & = -2 i \pi  \left(\frac{\sinh (2 \pi  t)}{\cosh (2 \pi  \sigma_s)+\cosh (2 \pi  t)}+1\right).
\end{align}
In particular, the saddle-point equation $f'(t) = 0$ takes the form
\begin{align}\label{Qsaddlepointeq}
\ln \big(1 + e^{2 \pi  (t-\sigma_s )}\big) + \ln \big(1 + e^{2 \pi (t + \sigma_s)}\big) = 2 \pi  \mu - \pi i.
\end{align}

Let us henceforth assume that $(\sigma_s,\mu) \in \mathbb R^2$ satisfy 
$$e^{2 \pi  \mu} > \sinh^2(2 \pi \sigma_s).$$ 
The next lemma shows that there is a unique saddle point in the strip $-1/2 < \im t < -1/4$.

\begin{lemma}\label{t0lemmaQ}
The saddle point equation (\ref{Qsaddlepointeq}) has a unique solution $t_0 = t_0(\sigma_s,\mu)$ in the strip $-1/2 < \im t < -1/4$ given by
\begin{align}\label{Qt0def}
t_0 = \frac{1}{2\pi} \ln\Big(-\cosh (2 \pi \sigma_s) - i \sqrt{e^{2 \pi  \mu} - \sinh^2(2 \pi \sigma_s)}\Big).
\end{align}
\end{lemma}
\begin{proof}
Since $\sigma_s,\mu \in \R$ satisfy $e^{2 \pi  \mu} > \sinh^2(2 \pi \sigma_s)$, we have $-1/2 < \im t_0 < -1/4$. 
On the other hand, letting $z = e^{2\pi t}$, the saddle-point equation (\ref{Qsaddlepointeq}) is
\begin{align}
\ln(1 + z e^{-2\pi \sigma_s}) + \ln(1 + z e^{2\pi \sigma_s}) = 2 \pi  \mu - \pi i.
\end{align}
Exponentiation gives the equation
$$(1 + ze^{-2\pi \sigma_s})(1 + z e^{2\pi \sigma_s}) = -e^{2 \pi  \mu},$$
which has two solutions $z_\pm$ given by
$$z_\pm = -\cosh (2 \pi \sigma_s) \pm i \sqrt{e^{2 \pi  \mu} - \sinh^2(2 \pi \sigma_s)},$$
where the expression inside the square root is $> 0$ by assumption. It follows that any solution of (\ref{Msaddlepointeq}) has the form $\frac{1}{2\pi} \ln z_\pm + in$ for some integer $n$. But 
$$\frac{\arg(z_-)}{2\pi} = - \frac{\arg(z_+)}{2\pi} = \frac{\arctan(\frac{\sqrt{e^{2\pi \mu} - \sinh^2(2\pi \sigma_s)}}{\cosh(2\pi \sigma_s)}) - \pi}{2\pi} \in (-1/2, -1/4),$$
so there can be at most one solution in the strip $-1/2 < \im t < -1/4$. 
It also follows that $t_0 = \frac{1}{2\pi} \ln z_-$ fulfills the saddle point equation (\ref{Qsaddlepointeq}) modulo an integer multiple of $2\pi i$, i.e.,
\begin{align}\label{Qsaddlemod}
\ln \big(1 + e^{2 \pi  (t_0-\sigma_s )}\big) + \ln \big(1 + e^{2 \pi (t_0 + \sigma_s)}\big) - 2 \pi  \mu + \pi i = 2\pi i N,
\end{align}
for some $N \in \Z$. Since 
\begin{align}\label{Qtsaddleimlog}
\im \ln \big(1 + e^{2 \pi  (t_0 \pm \sigma_s )}\big) \in (-\pi,0),
\end{align}
we conclude that the imaginary part of the left-hand side of (\ref{Qsaddlemod}) lies in the open interval $(-\pi,\pi)$, so the only possibility is $N = 0$. This shows that $t_0$ satisfies (\ref{Qsaddlepointeq}) and concludes the proof.
\end{proof}

We henceforth let $t_0$ be the solution (\ref{Qt0def}) of the saddle point equation.

\begin{lemma}\label{Qrefprimeprimelemma}
We have
\begin{align}\label{Qrefppnegative}
\re f''(t) < 0 \qquad \text{for all $t$ in the strip $-\frac{1}{2} < \im t < -\frac{1}{4}$.}
\end{align}
In particular, $\arg\big(-f''(t_0)\big) \in (-\pi/2, \pi/2)$.
\end{lemma}
\begin{proof}
Taking the real part of (\ref{Qd2fdt2}), we find, for all $t = t_1 + it_2$ with $-1/2 < t_2 < -1/4$,
\begin{align}\label{Qrefppt}
\re f''(t) = \frac{2\pi  ( \cosh (2 \pi  \sigma_s) \cosh (2 \pi t_1)+\cos (2 \pi  t_2))  \sin (2 \pi  t_2)}{(\cosh (2 \pi  (t_1-\sigma_s))+\cos (2 \pi  t_2))  (\cosh (2 \pi  (\sigma_s+t_1))+\cos (2 \pi  t_2))}.
\end{align}
Since $t_2 \in (-1/2, -1/4)$, we have $\cos(2\pi t_2) \in (-1,0)$ and $\sin(2\pi t_2) \in (-1,0)$, so the numerator in (\ref{Qrefppt}) is strictly negative and the denominator is strictly positive.
\end{proof}

We define $\alpha_0 = \alpha_0(\sigma_s,\mu)$ by
\begin{align}\label{Qalpha0def}
\alpha_0 = -\frac{\arg(-f''(t_0))}{2}.
\end{align}
In view of Lemma \ref{Qrefprimeprimelemma}, we have $\alpha_0 \in (-\pi/4,\pi/4)$.
We now proceed in the same way as in Section \ref{section4}: We first define an analytic bijection $\varphi$ satisfying (\ref{fvarphi}) from an open neighborhood $\mathcal{D}_\epsilon$ of $t_0$ onto $D_\epsilon(0)$ by (\ref{varphidef}). We then let $\Gamma = \Gamma_1 \cup \Gamma_2 \cup \Gamma_3$ be the union of $\Gamma_1 = \varphi^{-1}(-\epsilon) + \R_{\leq 0}$, $\Gamma_2 = \varphi^{-1}([-\epsilon, \epsilon])$, and $\Gamma_3 = \varphi^{-1}(\epsilon) + \R_{\geq 0}$.
Deforming the contour to $\Gamma$, equation (\ref{calQfg}) can be written as
\begin{align}\label{calQI1I2I3}
\mathcal{Q}\bigg(b, \frac{\sigma_s}{b}, \frac{\mu}{b}\bigg) 
& = \frac{e^{\frac{3\pi i}{4}}}{b} e^{\frac{f(t_0)}{b^2}} (I_1(b) + I_2(b) + I_3(b)) (1 + O(b^2)),
\end{align} 
where the functions $I_j(b) = I_j(b; \sigma_s, \mu)$ are defined by
$$I_j(b) = \int_{\Gamma_j} e^{\frac{f(t) - f(t_0)}{b^2}} g(t) dt, \qquad j = 1,2,3.$$

By (\ref{Qrefppnegative}), we have $\re f''(t) < 0$ everywhere in the strip $-1/2 < \im t < -1/4$, so for each $t_2 \in (-1/2, -1/4)$, $\re f(t_1 + it_2)$ is a strictly concave function of $t_1 \in \R$.
The following lemma therefore follows in the same way as Lemma \ref{Mreflemma}.

\begin{lemma}\label{Qreflemma}
 Shrinking $\epsilon > 0$ if necessary, the following estimate holds for all $t \in \Gamma_1 \cup \Gamma_3$ for some $c_0 > 0$:
$$\re \big[f(t) - f(t_0)\big] < -c_0.$$
\end{lemma}

The next lemma follows from Lemma \ref{Qefgatinftylemma} and Lemma \ref{Qreflemma}. It shows that $I_1(b)$ and $I_3(b)$ are exponentially small as $b \to 0$. The proof is analogous to the proof of Lemma \ref{MI1I3lemma} and is therefore omitted.

\begin{lemma}\label{QI1I3lemma}
There exists a $c > 0$ such that
$$I_1(b) = O(e^{-\frac{c}{b^2}}) \quad \text{and} \quad I_3(b) = O(e^{-\frac{c}{b^2}}) \qquad \text{as $b \to 0^+$}.$$
\end{lemma}

The next lemma follows in the same way as Lemma \ref{MI2lemma}.

\begin{lemma}\label{QI2lemma}
As $b \to 0^+$, it holds that
$$I_2(b) = \frac{g(t_0) \sqrt{2\pi}  }{ |f''(t_0)|^{1/2} e^{-i\alpha_0} }b
+ O(b^2).$$
\end{lemma}

Employing Lemma \ref{QI1I3lemma} and Lemma \ref{QI2lemma} in (\ref{calQI1I2I3}), we arrive at
\begin{align}\label{calQasymptotics}
\mathcal{Q}\Big(b, \frac{\sigma_s}{b}, \frac{\mu}{b}\Big) 
=&\; \frac{e^{\frac{3\pi i}{4}} g(t_0) \sqrt{2\pi}  }{ |f''(t_0)|^{1/2} e^{-i\alpha_0} } e^{\frac{f(t_0)}{b^2}} \big(1 + O(b)\big) \qquad \text{as $b \to 0^+$}.
\end{align}
It is possible to simplify the coefficient in (\ref{calQasymptotics}) as the next lemma shows. 

\begin{lemma}\label{Qegfsimplifiedlemma}
We have
\begin{align}\label{Qegfsimplified}
\frac{e^{\frac{3\pi i}{4}} g(t_0) \sqrt{2\pi}  }{ |f''(t_0)|^{1/2} e^{-i\alpha_0} } = \frac{e^{\frac{\pi i}{4}}}{\sqrt{2} (e^{2 \pi  \mu }-\sinh^2(2 \pi  \sigma_s))^{1/4}},
\end{align}
where $(e^{2 \pi  \mu }-\sinh^2(2 \pi  \sigma_s))^{1/4} > 0$.
\end{lemma}
\begin{proof}
At the saddle point $t_0$, we have
\begin{align*}
& f''(t_0) = -4 \pi  e^{-2 \pi  \mu } \cosh (2 \pi  \sigma_s) \sqrt{e^{2 \pi  \mu }-\sinh^2(2 \pi  \sigma_s)} + 4 \pi i  \left(e^{-2 \pi  \mu } \sinh^2(2 \pi  \sigma_s)-1\right).
\end{align*}
Moreover, since $\arg\big(-f''(t_0)\big) \in (-\pi/2, \pi/2)$ by Lemma \ref{Qrefprimeprimelemma}, we have $|f''(t_0)|^{1/2} e^{-i\alpha_0} = \sqrt{-f''(t_0)}$. 
On the other hand, substituting the expression (\ref{Qt0def}) for $t_0$ into the definition (\ref{Qgdef}) of $g$, we obtain 
$$g(t_0) = e^{-\pi  \mu } \sqrt{-\cosh (2 \pi  \sigma_s)-i \sqrt{e^{2
   \pi  \mu }-\sinh^2(2 \pi  \sigma_s)}},$$
where the principal branch is used for the square roots. It follows after simplification that
$$
\bigg(\frac{e^{\frac{3\pi i}{4}} g(t_0) \sqrt{2\pi}  }{ |f''(t_0)|^{1/2} e^{-i\alpha_0} }\bigg)^2 = 
\frac{e^{\frac{3\pi i}{2}} g(t_0)^2 2\pi }{-f''(t_0)} = 
\frac{i}{2 \sqrt{e^{2 \pi  \mu }-\sinh^2(2 \pi  \sigma_s)}},$$
where $e^{2 \pi  \mu }-\sinh^2(2 \pi  \sigma_s) > 0$ by assumption.
This shows that (\ref{Qegfsimplified}) holds at least up to sign. To fix the sign, note that $\arg g(t_0) \in (-\pi/2, -\pi/4)$; thus since $\alpha_0 \in (-\pi/4, \pi/4)$, it follows that
$$\arg \frac{e^{\frac{3\pi i}{4}} g(t_0) \sqrt{2\pi}  }{ |f''(t_0)|^{1/2} e^{-i\alpha_0} } \in \Big(0, \frac{3\pi}{4}\Big),$$
showing that (\ref{Qegfsimplified}) holds.  
\end{proof}

The function $f_{\mathcal{Q}}(\sigma_s, \mu)$ defined by $f_{\mathcal{Q}}(\sigma_s, \mu) := f(t_0; \sigma_s, \mu)$ can be written as in (\ref{finQ}). The asymptotic formula (\ref{Qfinalasymptotics}) then follows from (\ref{calQasymptotics}) and Lemma \ref{Qegfsimplifiedlemma}.

\subsection{$f_\mathcal{Q}$ and the Painlev\'e III$_3$ generating function}
We now turn to the proof of the last statement of Theorem \ref{Qth} which relates $f_\mathcal{Q}$ to the Painlev\'e III$_3$ generating function of Definition \ref{defgeneratingfuncPIII3}. The proof involves a careful analysis of the branches in the definition (\ref{defW}) of $W$ and certain identities for the classical dilogarithm.

\begin{lemma}
For all $(\sigma_s,\mu) \in \R^2$ such that $e^{2 \pi  \mu} > \sinh^2(2 \pi \sigma_s)$, we have
\begin{align}\nonumber
W\bigg(i\sigma_s, - \mu- \frac{i}{2}\bigg) 
= &\; \frac{\Li_2\Big(-e^{-2 \pi  (\mu +\sigma_s )} \big(\sinh (2 \pi  \sigma_s )-i \sqrt{e^{2 \pi  \mu }-\sinh
   ^2(2 \pi  \sigma_s )}\big)\Big)}{8 \pi
   ^2}
   	\\\nonumber
&   +\frac{\Li_2\Big(\frac{1}{e^{-2 \pi  \sigma_s } \big(\sinh (2 \pi  \sigma_s )-i \sqrt{e^{2 \pi  \mu }-\sinh^2(2 \pi 
   \sigma_s )}\big)}\Big)}{8 \pi ^2}     +\frac{\mu^2}{8} +\frac{i \mu}{8}-\frac{1}{32}
    	\\\label{Wonesaddleregime}
&
+ \frac{\ln^2\bigg(ie^{-\pi \mu}\Big( \sinh(2\pi \sigma_s) - i\sqrt{e^{2\pi \mu} - \sinh^2(2\pi \sigma_s)}\Big)\bigg)}{8\pi^2}.
\end{align}
where the principal branch is used for all functions. 
\end{lemma}
\begin{proof}
Define $\eta_{\alpha, \beta}(\sigma, \nu)$ by
\begin{align}\label{etaalphabetadef}
\eta_{\alpha, \beta}(\sigma, \nu) = \frac{\arcsin\big(e^{\pi (\nu-\beta \frac{i}{2})}\sin(2\pi e^{\frac{\pi i \alpha}{2}}\sigma)\big)}{2\pi}.
\end{align}
Note that $\eta_{0,0}(\sigma, \nu) = \eta(\sigma, \nu)$, where $\eta(\sigma, \nu)$
is the function in (\ref{etasigmanudef}).
The function $\arcsin{z} = -i \ln(\sqrt{1-z^2}+i z)$ is an analytic function of $z \in \C \setminus ((-\infty, -1] \cup [1, +\infty))$ with branch cuts along $(-\infty, -1]$ and $[1, +\infty)$. 
For $\alpha = \beta = 0$, $\sigma = 1/5$, and $\nu = -1$, we have $0<\eta_{\alpha,\beta}(\sigma, \nu) \approx  0.0065<\sigma<\frac{1}{4}$. Consequently, for $\sigma = 1/5$, $\nu = -1$, and all small enough $\alpha \geq 0$ and $\beta \geq 0$, we have by Definition \ref{defgeneratingfuncPIII3} that
\begin{align}\nonumber
8\pi^2 W\bigg(e^{\frac{\pi i \alpha}{2}}\sigma,\nu-\beta \frac{i}{2}\bigg) 
= &\; \Li_2\bigg( - e^{2i\pi(e^{\frac{\pi i \alpha}{2}}\sigma+\eta_{\alpha,\beta}(\sigma, \nu)-\frac{i (\nu-\beta \frac{i}{2})}2)}\bigg)  
	\\\nonumber
& + \Li_2\bigg( - e^{-2i\pi(e^{\frac{\pi i \alpha}{2}}\sigma+\eta_{\alpha,\beta}(\sigma, \nu)+\frac{i (\nu-\beta \frac{i}{2})}2)}\bigg) 
	\\ \label{Walphabeta}
& - (2\pi \eta_{\alpha,\beta}(\sigma, \nu))^2 + \pi^2 \bigg(\nu-\beta \frac{i}{2}\bigg)^2,
\end{align}
where principal branches are used. Let us analytically continue this formula to $\alpha = \beta = 1$.
As $\alpha$ increases from $0$ to $1$, $\im \eta_{\alpha,0}(1/5, -1)$ increases from $0$ to $\frac{1}{2 \pi }\arcsinh\left(e^{-\pi } \sinh \left(\frac{2 \pi }{5}\right)\right) \approx 0.011$. As we then increase $\beta$ from $0$ to $1$, $\arg \eta_{1,\beta}(1/5, -1)$ decreases back down to $0$. In particular, as we increase $\alpha$ from $0$ to $1$, and then $\beta$ from $0$ to $1$, $|\im \eta_{\alpha,\beta}(1/5, -1)|$ remains bounded by say $1/4$, so the arguments of the two dilogarithms in (\ref{Walphabeta}) do not cross the branch cut $[1, +\infty)$ of the dilogarithm because they stay inside the open unit disk:
\begin{align*}
& \bigg|- e^{\pm 2i\pi(e^{\frac{\pi i \alpha}{2}}\sigma+\eta_{\alpha,\beta}(\sigma, \nu) \mp \frac{i (\nu-\beta \frac{i}{2})}2)}\bigg|
= e^{\mp 2\pi \sigma \sin(\frac{\pi\alpha}{2}) \mp 2\pi \im \eta_{\alpha,\beta}(\sigma, \nu) + \pi \nu}
\leq 
 e^{\frac{2\pi}{5} + \frac{\pi}{2} - \pi} < 1
\end{align*}
for $\sigma = 1/5$ and $\nu = -1$.
Similarly, the argument $e^{\pi (\nu-\beta \frac{i}{2})}\sin(2\pi e^{\frac{\pi i \alpha}{2}}\sigma)$ of $\arcsin$ in (\ref{etaalphabetadef}) does not cross the branch cuts $(-\infty, -1]$ and $[1, +\infty)$, because it remains inside the open unit disk. This means that (\ref{Walphabeta}) is valid with principal branches also for $\alpha = \beta = 1$.
Thus, for $\sigma = 1/5$ and $\nu = -1$, we have
\begin{align*}
8\pi^2 W\bigg(i\sigma,\nu- \frac{i}{2}\bigg) 
= &\; \Li_2\bigg( - e^{2i\pi(i\sigma+\eta_{1,1}(\sigma, \nu) -\frac{i (\nu- \frac{i}{2})}2)}\bigg)  + \Li_2\bigg( - e^{-2i\pi(i\sigma+\eta_{1,1}(\sigma, \nu)+\frac{i (\nu- \frac{i}{2})}2)}\bigg) 
	\\
& - (2\pi \eta_{1,1}(\sigma, \nu))^2 + \pi^2 \bigg(\nu-\frac{i}{2}\bigg)^2,
\end{align*}
where the principal branch is used for the dilogarithms and for
$$\eta_{1,1}(\sigma, \nu) = \frac{\arcsin\left(e^{\pi  \nu } \sinh (2 \pi  \sigma )\right)}{2 \pi }
= \frac{1}{2\pi i}\ln\bigg(ie^{\pi \nu}\Big( \sinh(2\pi \sigma) - i\sqrt{e^{-2\pi \nu} - \sinh^2(2\pi \sigma)}\Big)\bigg).$$
Letting $\sigma_s = \sigma$ and $\mu = -\nu$ and simplifying, we find that (\ref{Wonesaddleregime}) holds for $\sigma_s = 1/5$ and $\mu = 1$.

Suppose that $(\sigma_s,\mu) \in \R^2$ satisfy $e^{2 \pi  \mu} > \sinh^2(2 \pi \sigma_s)$.
Then all square roots in (\ref{Wonesaddleregime}) are $> 0$. Moreover, as long as $e^{2 \pi  \mu} - \sinh^2(2 \pi \sigma_s) > 0$ and $\sigma_s, \mu \in \R$, the arguments of the dilogarithms in (\ref{Wonesaddleregime}) do not cross the branch cut $[1, +\infty)$, and the argument of $\ln^2$ in (\ref{Wonesaddleregime}) does not cross the branch cut $(-\infty,0]$. It follows that the formula (\ref{Wonesaddleregime}) is valid with principal branches for any $(\sigma_s,\mu) \in \R^2$ such that $e^{2 \pi  \mu} > \sinh^2(2 \pi \sigma_s)$.
\end{proof}

Let $(\sigma_s,\mu) \in \R^2$ be such that $e^{2 \pi  \mu} > \sinh^2(2 \pi \sigma_s)$.
By (\ref{Qt0def}), we have $z_- := e^{2\pi t_0} = -\cosh (2 \pi \sigma_s) - i \sqrt{e^{2 \pi  \mu} - \sinh^2(2 \pi \sigma_s)}$. Thus, letting $s := e^{2\pi \sigma_s} > 0$, we can write (\ref{Wonesaddleregime}) in terms of $z_-$ and $s$ as
\begin{align}\nonumber
W\bigg(i \sigma_s, - \mu- \frac{i}{2}\bigg) 
= &\; \frac{\Li_2\big(\frac{1}{s z_-+1}\big)}{8\pi ^2}
   +\frac{\Li_2\big(\frac{1}{\frac{z_-}{s}+1}\big)} {8 \pi ^2}
   + \frac{\mu ^2}{8}+\frac{i \mu}{8} -\frac{1}{32}
   	\\
&   +\frac{\ln^2\left(i e^{-\pi  \mu } (z_- + s)\right)}{8 \pi ^2},
\end{align}
Applying the identity (see \cite[Eq. 25.12.4]{NIST})
$$\Li_2(z) + \Li_2(z^{-1}) = -\frac{\pi^2}{6} - \frac{(\ln(-z))^2}{2}, \qquad z \in \C \setminus [1, +\infty),$$
with $z = s z_- + 1$ and $z = \frac{z_-}{s} + 1$, we deduce that
\begin{align*}\nonumber
W\bigg(i\sigma_s, - \mu- \frac{i}{2}\bigg) 
= &  -\frac{\Li_2(s z_-+1)}{8 \pi ^2} 
   -\frac{\Li_2\big(\frac{z_-}{s} + 1\big)}{8 \pi^2}
   +  \frac{\mu ^2}{8}+\frac{i \mu}{8}
   +\frac{\ln^2\left(i e^{-\pi  \mu } (s+z_-)\right)}{8 \pi ^2}
   	\\
   & 
   -\frac{\ln^2(-s z_- -1)}{16 \pi ^2}
   -\frac{\ln^2\left(-\frac{z_-}{s}-1\right)}{16 \pi ^2}
   -\frac{7}{96}.
\end{align*}
 Using the identity (see \cite[Eq. 25.12.6]{NIST})
\begin{equation} \label{Li2oneminusz}
\Li_2(z) + \Li_2(1 - z) = \frac{\pi^2}{6} - \ln(z) \ln(1-z)
\end{equation}
with $z = -z_- s$ and $z = -\frac{z_-}{s}$, we obtain
\begin{align*}\nonumber
W\bigg(i\sigma_s, - \mu- \frac{i}{2}\bigg) 
= &\; \frac{\Li_2(-s z_-)}{8 \pi ^2}
+ \frac{\Li_2(-\frac{z_-}{s})}{8 \pi^2} 
+ \frac{\mu ^2}{8}+\frac{i \mu }{8}
+\frac{\ln^2\left(i e^{-\pi  \mu } (s+z_-)\right)}{8 \pi^2} 
	\\
& -\frac{\ln^2(-s z_- -1)}{16 \pi ^2}
   -\frac{\ln^2\left(-\frac{z_-}{s}-1\right)}{16 \pi ^2}
   +\frac{\ln (-s z_-) \ln (s z_-+1)}{8 \pi ^2}
   	\\
&  +\frac{\ln (-\frac{z_-}{s}) \ln (\frac{z_-}{s} +1)}{8 \pi ^2}
   -\frac{11}{96}.
 \end{align*}
 Now $z_-$ lies in the third quadrant and $s > 0$. Thus, $\frac{z_-}{s} + 1$ and $z_- s + 1$ lie in the lower half-plane, so
\begin{align*}
& \ln\left(i e^{-\pi \mu } (s + z_{-})\right)
= \frac{\pi i}{2} - \pi \mu +  \ln(s) + \ln\left(\frac{z_{-}}{s}+1\right),
	\\
& \ln\left(\frac{z_{-}}{s} + 1\right) = \ln\left(-\frac{z_{-}}{s} - 1\right) - \pi i, \qquad
\ln(s z_{-}  + 1) = \ln(-s z_{-} - 1) - \pi i,
\end{align*}
Furthermore, 
$$\ln(-s z_- - 1) = 2\pi \mu - \ln(- \frac{z_-}{s} - 1) + \pi i,$$
because $(z_- s +1)(\frac{z_-}{s}+1) = -e^{2 \pi  \mu}$.
Using these identities, we obtain
\begin{align*}\nonumber
4\pi i W\bigg(i\sigma_s, - \mu- \frac{i}{2}\bigg) 
= &\;
-\frac{\Li_2(-s z_-)}{2 \pi i}
-\frac{\Li_2\left(-\frac{z_-}{s}\right)}{2 \pi i}
- \frac{\ln^2(s)}{2 \pi i}
 +\bigg(\frac{1}{2}+i \mu \bigg)\ln (-z_-)
-\frac{i \pi }{3}.
\end{align*}
Comparing with (\ref{finQ}), we arrive at (\ref{fWrelationexpected}). This concludes the proof of \eqref{fWrelationexpected} and of Theorem \ref{Qth}.

\appendix

\section{Semiclassical limit of the difference equations} \label{appendix}
It was shown in \cite{LR24} that the functions $\mathcal M$ and $\mathcal Q$ satisfy two pairs of difference equations. In this appendix, we explain how the differential equations \eqref{diffeqfzeta} and \eqref{diffeqfomega} of Proposition \ref{prop:diffeqforf} formally arise from the difference equations satisfied by the function $\mathcal M$. This is the approach we originally followed to guess a relation between the semiclassical limit of $\mathcal M$ and the Painlev\'e I generating function. 
The case of the function $\mathcal Q$ is similar.

We do not attempt to make the arguments in this appendix rigorous; the presentation serves merely as motivation for the constructions.

Define two difference operator $D_\mathcal{M}(b,\zeta)$ and $\tilde{D}_\mathcal{M}(b,\omega)$ by
\begin{align} \label{HM}
& D_\mathcal{M}(b,\zeta) = e^{2\pi b \zeta} \lb 1 - e^{ib\partial_\zeta} \rb, \\
& \tilde{D}_\mathcal{M}(b,\omega) = e^{2\pi b \omega} - 2 e^{\pi b (\omega-\frac{ib}2)} \cosh(\pi b(\tfrac{ib}2+\omega)) e^{ib \partial_\omega}.
\end{align}
For $\im \omega<(b+b^{-1})/2$ and $\im(\zeta+\omega)>0$, the function $\mathcal{M}$ satisfies \cite{LR24}
\begin{subequations}\label{diffeqM}\begin{align}
\label{diffeqM1} & D_\mathcal{M}(b,\zeta) \mathcal{M}(b,\zeta,\omega) = e^{-2\pi b \omega} \mathcal{M}(b,\zeta,\omega), \\
\label{diffeqM2} & D_\mathcal{M}(b^{-1},\zeta) \mathcal{M}(b,\zeta,\omega) = e^{-2\pi b^{-1} \omega} \mathcal{M}(b,\zeta,\omega),
\end{align}\end{subequations}
while for $\im(\omega+ib^{\pm 1}) < (b+b^{-1})/2$ and $\im(\zeta+\omega)>0$ it satisfies 
\begin{subequations}\label{diffeqtildeM}\begin{align}
\label{diffeqtildeM1} & \tilde{D}_\mathcal{M}(b,\omega)\mathcal{M}(b,\zeta,\omega) = e^{-2\pi b \zeta} \mathcal{M}(b,\zeta,\omega), \\
\label{diffeqtildeM2} & \tilde{D}_\mathcal{M}(b^{-1},\omega)\mathcal{M}(b,\zeta,\omega) = e^{2\pi b^{-1} \zeta} \mathcal{M}(b,\zeta,\omega).
\end{align}\end{subequations}

Suppose now that $\mathcal{M}(b,\zeta,\omega)$ admits semiclassical asymptotics of the form
\begin{align}
\label{Fsc} & \mathcal{M}\lb b,\frac{\zeta}b,\frac{\omega}b\rb = g_\mathcal{M}(\zeta,\omega) e^{\frac{f_\mathcal{M}(\zeta,\omega)}{b^2}}(1+O(b)) \qquad \text{as} \; b \to 0^+. 
\end{align}
Assuming that we can apply the operators $e^{ib^2\partial_\zeta}$ and $e^{ib^2\partial_\omega}$ to the above equation without increasing the error term, we arrive at
\begin{align}
\label{Fzetasc} & e^{ib^2\partial_\zeta} \mathcal{M}\lb b,\frac{\zeta}b,\frac{\omega}b\rb = \left[e^{ib^2 \partial_\zeta}\lb  g_\mathcal{M}(\zeta,\omega) e^{\frac{f_\mathcal{M}(\zeta,\omega)}{b^2}}\rb\right] (1+O(b)) \qquad \text{as} \; b \to 0^+, \\
\label{Fomegasc} & e^{ib^2\partial_\omega} \mathcal{M}\lb b,\frac{\zeta}b,\frac{\omega}b\rb = \left[e^{ib^2 \partial_\omega}\lb  g_\mathcal{M}(\zeta,\omega) e^{\frac{f_\mathcal{M}(\zeta,\omega)}{b^2}}\rb\right] (1+O(b)) \qquad \text{as} \; b \to 0^+.
\end{align}
For $x\in \{\zeta,\omega\}$, we have
$$e^{i b^2 \partial_x} \left( g_\mathcal{M}(\zeta,\omega) e^{\frac{f_\mathcal{M}(\zeta,\omega)}{b^2}} \right) = \sum_{k=0}^\infty \frac{(i b^2 \partial_x)^k}{k!} \left[g_\mathcal{M}(\zeta,\omega) e^{\frac{f_\mathcal{M}(\zeta,\omega)}{b^2}} \right].$$
Keeping only the dominant contributions from the derivatives, we infer that
\begin{align} \label{derivsc}
e^{i b^2 \partial_x}  \mathcal{M}\lb b,\frac{\zeta}b,\frac{\omega}b\rb = \lb e^{i \partial_x f_\mathcal{M}(\zeta,\omega)}\rb  g_\mathcal{M}(\zeta,\omega) e^{\frac{f_\mathcal{M}(\zeta,\omega)}{b^2}} \lb 1 + O(b) \rb \qquad \text{as} \; b \to 0^+.
\end{align}
The semiclassical limit of the difference equations \eqref{diffeqM1} and \eqref{diffeqtildeM1} can now readily be computed. More precisely, as $b\to 0^+$ we obtain 
\begin{align}
\label{expipartialzetaf} & e^{i \partial_\zeta f_\mathcal{M}(\zeta,\omega)} = - e^{-2\pi(\zeta+\omega)} \lb 1 - e^{2\pi(\zeta+\omega)}\rb, \\
\label{expipartialomegaf} & e^{i\partial_\omega f_\mathcal{M}(\zeta,\omega)} = -e^{-2\pi \zeta} \frac{(1-e^{2\pi(\zeta+\omega)})}{1+e^{2\pi \omega}}.
\end{align}
Equations \eqref{expipartialzetaf} and \eqref{expipartialomegaf} are equivalent to the exponentials of \eqref{diffeqfzeta} and \eqref{diffeqfomega}, respectively. 

Let us finally mention that the case of $\mathcal Q$ is similar but slightly more complicated. The reason is that the function $\mathcal Q(b,\sigma_s,\mu)$ satisfies a pair of difference equations in the variable $\mu$ where the difference operators are second order (see \cite[Equation (7.30)]{LR24}). Therefore, the formal semiclassical limit of these equations formally leads to a second order algebraic equation of the form $\alpha X^2 + \beta X + \gamma = 0$ where $X=e^{i\partial_\mu f_{\mathcal Q}}$.

\end{document}